\documentclass{amsart}
\usepackage{latexsym,amssymb}
\usepackage[english]{babel}

\usepackage{amsfonts, amsthm}

\def\Z{\mathbb{Z}}

\def\BHR{\mathop{\rm BHR}}

\newcommand{\ext}[1]{t^{#1}\textrm{-extendable} }
\newcommand{\fr}[1]{\stackrel{#1}{\rightharpoondown}}

\newtheorem{defi}{Definition}[section]
\newtheorem{prop}[defi]{Proposition}
\newtheorem{lem}[defi]{Lemma}
\newtheorem{rem}[defi]{Remark}

\newtheorem{thm}[defi]{Theorem}
\newtheorem*{nota}{Notation}
\newtheorem*{conj}{Conjecture}

\theoremstyle{definition}

\newtheorem{ex}[defi]{Example}

\begin{document}

\title[ On the Buratti--Horak--Rosa Conjecture\ldots]{On the Buratti--Horak--Rosa Conjecture about Hamiltonian Paths in Complete Graphs}

\author{Anita Pasotti}
\address{DICATAM - Sez. Matematica, Universit\`a degli Studi di Brescia, Via
Branze 43, I-25123 Brescia, Italy}
\email{anita.pasotti@unibs.it}

\author{Marco Antonio Pellegrini}
\address{Dipartimento di Matematica e Fisica, Universit\`a Cattolica del Sacro Cuore, Via Musei 41,
I-25121 Brescia, Italy}
\email{marcoantonio.pellegrini@unicatt.it}


\begin{abstract}
In this paper we investigate a problem proposed by Marco Buratti, Peter Horak and Alex
Rosa
(denoted by BHR-problem) concerning Hamiltonian paths in the complete graph with
prescribed edge-lengths.
In particular we solve $\BHR(\{1^a,2^b,$ $t^c\})$ for any
even integer $t\geq4$, provided that $a+b\geq t-1$.
Furthermore, for $t=4,6,8$
we present a complete solution of $\BHR(\{1^a,2^b,t^c\})$ for any positive integer
$a,b,c$.
\end{abstract}

 \maketitle

\section{Introduction}
Throughout this paper $K_v$ will denote the complete graph on $\{0,1,\dots,v-1\}$ for any positive integer $v$.
For the basic terminology on graphs we refer to \cite{Wbook}.
Following \cite{HR}, we define the {\it length} $\ell(x,y)$ of an edge $[x,y]$ of $K_v$ as
$$\ell(x,y)=min(|x-y|,v-|x-y|).$$
If $\Gamma$ is any subgraph of $K_v$, then the list of edge-lengths of $\Gamma$ is the multiset
$\ell(\Gamma)$ of the lengths (taken with their respective multiplicities) of  all the edges
of $\Gamma$. For our convenience, if a list $L$ consists of
$a_1$ $1'$s, $a_2$ $2'$s, \ldots, $a_t$ $t'$s,
we will write $L=\{1^{a_1},2^{a_2},\ldots,t^{a_t}\}$.

The following conjecture \cite{BM, W} is due to Marco Buratti
(2007, communication to Alex Rosa).

\begin{conj}[Buratti]
For any prime $p=2n+1$
and any multiset $L$ of $2n$ positive integers not exceeding $n$, there exists a
Hamiltonian path $H$ of $K_p$ with $\ell(H)=L$.
\end{conj}

The conjecture is almost trivially true
in the case that $L$ has just one edge-length, but, in general, the problem seems to be
very difficult.
The case of exactly two distinct edge-lengths
 has been solved independently in \cite{DJ, HR} and
Mariusz Meszka checked that the conjecture is true for all primes $p\leq23$ by computer.
Some  general results on the conjecture can be found in \cite{HR}, in particular that it
is true when there is an edge-length occurring ``sufficiently many times'' in $L$.

In \cite{HR} Peter Horak and Alex Rosa generalized Buratti's conjecture.
Such a generalization has been restated in an easier form in \cite{1235}, as follows.

\begin{conj}[Horak and Rosa]\label{HR iff B}
Let $L$ be a list of $v-1$ positive integers not exceeding $\big\lfloor \frac{v}{2}\big\rfloor$.
Then there exists a Hamiltonian path $H$ of $K_v$ such that $\ell(H)=L$ if, and only if, the following condition holds:
\begin{equation}\label{B}
\left. \begin{array}{c}
\textrm{for any divisor $d$ of $v$, the number of multiples of $d$} \\
\textrm{appearing in $L$ does not exceed $v-d$.}
\end{array}\right.
\end{equation}
\end{conj}

Following \cite{1235}, by $\BHR(L)$ we will denote the above conjecture for a given list $L$.
The case of exactly three distinct edge-lengths has been
solved when these lengths are $1,2,3$ in \cite{CDF} and when they are
$1,2,5$ or $1,3,5$ or $2,3,5$ in \cite{1235}.
More in general, in \cite{1235} we proved that
$\BHR(\{1^a,2^b,3^c,5^d\})$ holds for all integers $a,b,c,d\geq0$
finding, in such a way, the first set $S$ of size four for which we can say that
$\BHR(L)$ is true when the underlying-set of the list $L$ is $S$.\\
In the same paper we have also shown how this problem is related to the existence
of cyclic graph decompositions. In detail, we proved that $\BHR(L)$ can be reformulated as follows.

\begin{conj}
A Cayley multigraph $Cay[\Z_v:\Lambda]$ admits a cyclic decomposition into Hamiltonian paths if and only if $\Lambda=L \ \cup \ -L$
with $L$ satisfying condition $(\ref{B})$.
\end{conj}

For reader convenience we recall the definition of a Cayley multigraph, see \cite{BCD,BMnew}.
Given an additive group $G$ and a list $\Lambda \subseteq G\setminus\{0\}$ such that for any $g\in\Lambda$
we also have $-g\in\Lambda$, the \emph{Cayley multigraph} on $G$ with connection multiset $\Lambda$,
denoted by $Cay[G:\Lambda]$, is the graph with vertex set $G$ and where the multiplicity of an edge
$[x,y]$ is the multiplicity of $x-y$ in $\Lambda$.
We point out that if $G=\mathbb{Z}_v$ then  $Cay[\mathbb{Z}_v:\Lambda]$
is also called a \emph{circulant multigraph}.

In this paper we want to describe a general strategy to solve $\BHR(L)$  when the underlying-set of $L$ is $\{1,2,t\}$ for any arbitrary even integer $t\geq 4$.
\bigskip

All the known results about $\BHR(L)$-problem have been obtained thanks to \emph{cyclic}
and
\emph{linear realizations}.
A \emph{cyclic realization} of a list $L$ with $v-1$ elements each
from the set
$\{1,\ldots,\lfloor\frac{v}{2}\rfloor\}$
is a Hamiltonian path $P$
of $K_{v}$ such that the multiset of edge-lengths
of $P$ equals $L$.
Hence, it is clear that
$\BHR(L)$
can be also formulated as: every such a list $L$ has a cyclic realization
if and only if condition (\ref{B}) is satisfied.
For example, the path $[0,6,5,1,9,7,3,2,8,4]$ is a cyclic realization of
$L=\{1^2,2^2,4^5\}$.\\
A \emph{linear realization} of a list $L$ with $v-1$ positive integers not exceeding
$v-1$ is a Hamiltonian path
$[x_0,x_1,\ldots,x_{v-1}]$ of $K_v$
such that $L=\{|x_i-x_{i+1}|\ |\ i=0,\ldots,v-2\}$.
For instance, one can easily check that the path $[0,8,7,6,4,2,10,9,1,3,5]$ is a
linear realization
of $L=\{1^3,2^4,8^3\}$.\\
We denote by $cL$ and $rL$
a cyclic and a linear realization of $L$, respectively.
In this paper we shall choose $0$ as first vertex of any path.

\begin{rem}\label{cyclin}
Every linear realization of a list $L$ can be viewed as a
cyclic realization
of a suitable list $L'$, but not necessarily of the same list.
Anyway if all the elements in the list are less than or equal to $\lfloor\frac{|L|+1}{2}\rfloor$, then every linear realization of $L$ is
also a cyclic realization of the same list $L$ $($see Section $3$ of \cite{HR}$)$.
\end{rem}

In Section 2 we introduce two new classes of linear realizations and we show how
they are fundamental to solve $\BHR(L)$-conjecture, in particular, when an element of $L$
is not a fixed number as in the case investigated in this paper.
In Section 3 we construct linear realizations of $L=\{1,2^b,t^c\}$, while in Section 4,
using the results of previous sections,
we construct linear realizations for the general case $L=\{1^a,2^b,t^c\}$
with $a\geq2$. Finally, in Section 5 we present a complete solution of $\BHR(\{1^a,2^b,t^c\})$,
for $t=4,6,8$.\\
The main result of the paper is the following.

\begin{thm}\label{main}
If $t\geq 4$ is an even integer, then $\BHR(\{1^a,2^b,t^c\})$ holds for any
$a,b,c\geq 1$ with $a+b\geq t-1$.
\end{thm}

\section{Special linear realizations}

In this section we introduce two particular kinds of linear realizations which turned out
to
be very useful for solving $\BHR(\{1^a,2^b,t^c\})$.

\begin{defi}
Given a list $L$, we will say that a linear realization $rL$ is
\begin{itemize}
\item \emph{special of type $1$}
if $|L|$ and $|L|-1$ are adjacent in $rL$;
\item \emph{special of type  $2$} if
its endpoints are $0$ and $1$.
\end{itemize}
For short, we will denote by $S_1L$ and $S_2L$ a special linear realization of $L$
of type $1$ and $2$, respectively. Also, by $S_{1,2}L$ we will mean a special linear realization
of $L$ of both type $1$ and $2$.
\end{defi}

Now we will show how starting from a special linear realization of a given list $L$
it is possible to obtain a special linear realization of an infinite class of lists.

\begin{lem}\label{special1}
If a list $L$ admits a special linear realization of type $1$, then also the list
$L'=L\cup\{2^b\}$ admits a special linear realization of type $1$, for any positive integer $b$.
\end{lem}
\begin{proof}
Set $v=|L|+1$ and let $S_1L=[0,\ldots, v-1, v-2,\ldots]$. Consider the realization obtained from
$S_1L$ by adding the element $v$ between $v-1$ and $v-2$, namely
$S_1L'=[0,\ldots, v-1, v, v-2,\ldots]$. It is easy to see that $S_1L'$
is a special linear realization of type $1$ of $L'=L\cup \{2\}$.
In order to obtain a $S_1(L\cup\{2^b\})$ it is sufficient to apply
the above process $b$ times.
If $S_1L=[0,\ldots, v-2, v-1,\ldots]$ the thesis can be obtained in the same way.
\end{proof}

\begin{ex}
Consider $S_1\{1,2^{15},16\}=[0,2,4,6,8,10,12,14,16,17,1,3,5,7,9,$ $11,13,15]$.
Following the proof of Lemma $\ref{special1}$ if we are looking, for instance, for
$S_1\{1,2^{18},16\}$ we have to apply three times the process,  obtaining
$[0,2,4,6,8,10, 12,$ $14,16,
\mathbf{18,20, 19},17,1,3,5,7, 9,11,13,15]$.
\end{ex}

\begin{lem}\label{special2}
If a list $L$ admits a special linear realization of type $2$, then also the list
$L'=L\cup\{2^b\}$ admits a special linear realization of type $2$, for any integer $b\geq2$.
\end{lem}
\begin{proof}
Let $S_2L=[0,x_1,x_2,\ldots,x_n,1]$ and consider
$S_2L'=[0,2,x_1+2,x_2+2,\ldots,x_n$ $+2,3,1]$,
obtained from $S_2L$ by prepending $0$ and appending $1$ to the translate of $S_2L$ by $2$.
It is easy to see that $S_2L'$ is a special linear realization of type $2$ of
$L'=L\cup\{2^2\}$.
So, if $b$ is even, in order to obtain a $S_2(L\cup\{2^b\})$ it is sufficient
to repeat the process $\frac{b}{2}$ times.
Now consider $S_2L''=[0,2,4,x_n+3,\ldots,x_2+3,x_1+3,3,1]$,
obtainable via the following steps: translate $S_2L$ by $3$, prepend $1$, append $2,0$ and reverse.
It is not hard to see that
$S_2L''$ is a special linear realization of type $2$ of $L''=L\cup\{2^3\}$.
Hence, if $b$ is odd, say $b=2k+3$, in order to obtain a $S_2(L\cup\{2^b\})$ it is
sufficient
to apply $k$ times the first process and once the second process.
\end{proof}

\begin{ex}
Starting from $S_2\{1,2^6,8^2\}=[ 0, 8, 6, 4, 2, 3, 5, 7, 9, 1 ]$
we can obtain, for example,  $S_2\{1, 2^8,8^2\}=[ 0, 2, 10, 8, 6, 4, 5, 7, 9, 11, 3, 1 ]$
and $S_2\{1,2^9,8^2\}=[ 0, 2, 4, 12, 10, 8, 6,$ $ 5, 7, 9, 11, 3, 1 ]$.
\end{ex}

\begin{lem}\label{special1e2}
If there exists a special linear realization of type $1$ of a list $L_1$
and a special linear realization of type $2$ of a list $L_2$,
then there exists a linear realization of $L_1\cup L_2$.
\end{lem}
\begin{proof}
Set $v=|L_1|+1$.
Let $S_1L_1=[0,y_1,\ldots,y_r,v-2,v-1,z_1,\ldots,z_s]$ and
$S_2L_2=[0,x_1,x_2,\ldots,x_n,1]$.
Now consider
$rL'=[0,y_1,\ldots,y_r,v-2,x_1+v-2,x_2+v-2,\ldots,x_n+v-2,v-1,z_1,\ldots,z_s]$.
One can easily check that $rL'$ is a linear realization of $L'=L_1\cup L_2$.\\
If $S_1L_1=[0,y_1,\ldots,y_r,v-1,v-2,z_1,\ldots,z_s]$ the thesis can be obtained in a
similar way.
\end{proof}

\section{Construction of linear realizations of $\{1, 2^b,t^c\}$} \label{3}

First of all we make some remarks which will be very useful in the following.
Since the case in which the list has exactly two elements has been completely solved in \cite{DJ}
and \cite{HR}, we consider through all the paper  $\{1^a, 2^b,t^c\}$ with
$t\geq 4$ an even integer and $a,b,c\geq1$.

\begin{rem}\label{RemA>a}
If there exists a linear realization $rL=[0,x_1,x_2,\ldots,x_s]$ of $L=\{1^a,2^b,t^c\}$,
then we have a linear realization of $L'=\{1^{a+A},2^b,t^c\}$ for any $A\geq 0$. In fact,
it suffices to consider $rL'=[0,1,2,\ldots,A,x_1+A,x_2+A,\ldots,x_s+A]$.
\end{rem}

\begin{lem}
If a list $L=\{1^{a_1},2^{a_2}, \ldots, t^{a_t}\}$ admits a linear realization $rL$, then $a_i+i-1\leq |L|$ for all $i=1,\ldots,t$.
\end{lem}

\begin{proof}
By way of contradiction, suppose that there exists an integer $j$ such that $a_j+j -1 > |L|$. This means that $\sum_{i\neq j} a_i = |L|-a_j < j-1$. Remove from the path $rL$ the edges of length $i$ with $i\neq j$. In such a way we obtain at most  $ j-1$ connected components whose vertices belong to the same congruence class modulo $j$. Hence, we have a partition of the $v=|L|+1$ elements of $\Z_{v}$  in at most $j-1$ congruence classes modulo $j$, which is clearly an absurd.
\end{proof}

 \begin{rem}\label{a+b+c}
In view of previous lemma, since we are looking for \emph{linear} realizations of $L=\{1^a,2^b,t^c\}$, in the following we will suppose
$a+b\geq t-1$.
 \end{rem}

\begin{nota}
In order to present our realizations in a short way we will always use the notation
here explained.
In every realization of a list $L=\{1^a,2^b,t^c\}$ the symbol  ``$x\stackrel{+ i}{\rightharpoondown}y$'' means the arithmetic progression $x,x+i,
x+2i,\ldots,
y-i,y$, where $x<y$ are congruent modulo $i$.
Analogous meaning for the symbol
 ``$x\stackrel{-i}{\rightharpoondown}y$'', when $x > y$.
If $i=t$, for short, by ``$x\rightarrow y$'' we will mean the arithmetic progression $x,x+t,
x+2t,\ldots,
y-t,y$
or the arithmetic progression $x,x-t, x-2t, \ldots, y+t,y$ according to whether $y > x$ or
$y
< x$,
respectively.
\end{nota}

In view of Remark \ref{RemA>a} it is natural to start investigating the case $a=1$,
namely we construct linear realizations of $\{1, 2^b,t^c\}$. In order to reduce the
number of realizations we need to describe, we introduce a particular class of linear
realizations and two related lemmas.

\begin{defi}
Given $L=\{1^a, 2^b, t^c\}$, set $v=|L|+1$, $e_1=[v-t,v-t+2]$ and
$e_2=[v-t,v-t+1]$. A linear realization $rL$ is $t^{4w}$-\emph{extendable}
if, for $i=1$ or $i=2$, both $e_i+4x$ and $e_i+4x+i$ are edges of $rL$ for any $x=0,\ldots,w$.
\end{defi}

\begin{lem}\label{c,c+4w}
If there exists a $t^{4w}$-extendable linear realization $rL$ of $L=\{1^a, 2^b, t^c\}$, then
there exists a linear realization  of $L'=\{1^a, 2^b, t^{c+4x+4}\}$ for any
$x=0,\ldots,w$. Furthermore, if $rL$ is special of type $2$, also $rL'$ is special of type $2$.
\end{lem}

\begin{proof}
If $e_1$ and $e_1+1$ are edges of
$rL$, we replace  $e_1$ with $[v-t,v,v+2,v-t+2]$
and $e_1+1$ with $[v-t+1,v+1,v+3,v-t+3]$. If $e_2$ and $e_2+2$
are edges of
$rL$, we replace $e_2$ with $[v-t,v,v+1,v-t+1]$
and $e_2+2$ with $[v-t+2,v+2,v+3,v-t+3]$.
In this way, we obtain a linear realization of $\{1^a, 2^b, t^{c+4}\}$.
Now, we reapply this process to the edges $e_1+4$ and $e_1+5$ (or to
$e_2+4$ and $e_2+6$, respectively), obtaining a linear realization of $\{1^a,
2^b, t^{c+8}\}$.
Applying this process $x+1$ times (with $0\leq x\leq w$)
we obtain a linear realization of $\{1^a, 2^b,
t^{c+4x+4}\}$.
\end{proof}

\begin{ex}
It is easy to see that $S_1L=[0\fr{+2} 14,15,1\fr{+2} 13]$
is a $\ext{8}$ linear realization of $L=\{1,2^{13},14\}$ (consider the edges
$[2,4]$ and $[3,5]$; $[6,8]$ and $[7,9]$;  $[10,12]$ and $[11,13]$).
Applying once the construction illustrated in the proof of
Lemma $\ref{c,c+4w}$,
we obtain
$rL'= [0, 2, \mathbf{16,18}, 4\fr{+2} 14,15,1,3,\mathbf{17,19}, 5 \fr{+2} 13]$
which is a realization of $L'=\{1,2^{13},14^{5}\}$.
For example, if we are looking for a linear realization of $L''=\{1,2^{13},14^{9}\}$
we have to apply the process to $rL'$:
$rL''= [0, 2, 16,18, 4, 6,\mathbf{20,22}, 8\fr{+2} 14,15,1,3, 17,19, 5,
7,\mathbf{21, 23}, 9 \fr{+2} 13]$. Finally, a linear realization of the list
$L'''=\{1,2^{13},14^{13}\}$ is $rL'''= [0, 2, 16,  18, 4, 6, 20,22, 8, 10,
\mathbf{24},$ $\mathbf{26},12, 14,15,1,3, 17, 19, 5,
7, 21,$ $ 23, 9,11, \mathbf{25,27}, 13]$.
\end{ex}

\begin{lem}\label{senzak}
Let $t$ be even and suppose that $rL$ is a $\ext{4w}$ linear realization of a list
$L=\{1^a,2^b,t^c\}$  for $w=\lfloor\frac{t}{4}\rfloor-1$.
Then, we have the following:
\begin{itemize}
\item[\rm{(i)}] If $t\equiv 0\pmod 4$, then for any $k\geq1$ there exists a $\ext{4w}$
linear realization of $L'=\{1^a,2^b,t^{tk+c}\}$.
\item[\rm{(ii)}] If $t \equiv 2\pmod 4$ and $rL$ is special of type $1$, then for any $k\geq 1$ there exists
a $\ext{4w}$ special linear realization of $L'=\{1^a,2^b,t^{tk+c}\}$ of type $1$.
\item[\rm{(iii)}] If $rL$ is special of type $1$, then for any $y\geq 1$ there
exists a $\ext{4w}$ special realization of
$L'=\{1^a, 2^{b+4y}, t^c\}$ of type $1$.
\end{itemize}
\end{lem}

\begin{proof}
To prove (i) and (ii), clearly it suffices to consider the case $L'=\{1^a, 2^{b},t^{t+c}\}$.
Let $t=4u$ or $4u+2$, so $w=u-1$. Since $rL$ is $\ext{4w}$, the
edges $[v-t,v-t+2], [v-t+1,v-t+3],\ldots,[v-t+4u-4,v-t+4u-2],[v-t+4u-3,v-t+4u-1]$
are edges of $rL$, where
$v=|L|+1$. Applying the process described in Lemma \ref{c,c+4w} $w$ times, we obtain a
linear realization of $L''=\{1^a,2^{b},t^{c+4u}\}$, which contains the following subpaths:
$[v-t,v, v+2, v-t+2], [v-t+1, v+1, v+3,v-t+3],\ldots,[v-t+4u-4,v+4u-4, v+4u-2,
v-t+4u-2], [v-t+4u-3,v+4u-3, v+4u-1, v-t+4u-1]$. Thus, 
$[v+4x,v+2+4x]$ and $[v+1+4x,v+3+4x]$ are edges of $rL''$ for any $x=0,\ldots,w$. Hence,
$rL''$ is $\ext{4w}$.\\
If $t=4u$, taking $L'=L''$, we prove (i).
If $t=4u+2$, replace in $rL''$
either
$[v-2,v-1]$ by
$[v-2, v+t-2, v+t-1, v-1]$ or $[v-1, v-2]$ by $[v-1, v+t-1, v+t-2, v-2]$ (this is
possible, since we are assuming that $rL$ is special of type $1$). In this way
we obtain a special linear realization of type $1$ of $L'=\{1^a,2^{b},t^{c+4u+2}\}$, proving
(ii).

To prove (iii), clearly it suffices to consider the case $L'=\{1^a, 2^{b+4},$
$t^c\}$.
Assume $S_1L=[0,y_1,\ldots,y_r, v-2, v-1,z_1,\ldots,z_s]$, where again $v=|L|+1$, and
let $S_1L'$ be the realization obtained from $S_1L$ following
the proof of Lemma \ref{special1}, namely
$S_1L'=[0,y_1,\ldots,y_r,v-2,v,v+2,v+3,v+1,v-1,z_1,\ldots,z_s]$; clearly $|L'|=v+3$.
Since $S_1L$ is $\ext{4u-4}$, the edges $[v-t,v-t+2], [v-t+1,v-t+3], \ldots,[v-t+4(u-1),v-t+2+4(u-1)], [v-t+1+4(u-1),v-t+3+4(u-1)] $
are edges of $S_1L$.
So $S_1L'$ contains the edges $[v+4-t,v+4-t+2],[v+4-t+1,v+4-t+3], \ldots,[v-t+4(u-1),v-t+2+4(u-1)], [v-t+1+4(u-1),v-t+3+4(u-1)]
$.
If $t=4u$, also the edges $[ v,v+2]$ and $[v+1,v+3]$ appears in $S_1L'$; while
if $t=4u+2$ the
edges $[v-2,v]$ and $[v-1,v+1]$ are edges of $S_1L'$. So, also $S_1L'$ is $\ext{4u-4}$.

If the edges of $rL$ are $[v-t+4x,v-t+1+4x]$ and $[v-t+2+4x,v-t+3+4x]$, the
statements can be proved in a  similar way.
\end{proof}

Now we construct linear realizations of $L=\{1,2^b,t^c\}$.
Writing $c=tk+d$ with $k\geq 0$ and $0\leq d<t$, we distinguish two cases, according to
the parity of $d$.

\begin{prop}\label{c even}
Let $t\geq 4$ and $d\geq 0$ be even integers.
The list $\{1, 2^b, t^{tk+d}\}$ admits a special linear realization
of type $2$
 for any $b\geq t-1$ if $d\equiv 2 \pmod4$ and $t\equiv2\pmod4$;
 for any $b\geq t-2$ otherwise.
\end{prop}

\begin{proof}
Suppose firstly $t \equiv0\pmod 4$ and consider the following special linear
realizations of type $2$:
\begin{footnotesize}
$$\left.\begin{array}{rl}
S_{1,2}\{1,2^{t-2}\}=\hspace{-0.3cm} & [0\fr{+2} t-2, t-1\fr{-2} 1],\\
S_{1,2}\{1, 2^{t-1}\} =\hspace{-0.3cm} & [ 0\fr{+2} t, t-1\fr{-2} 1],\\
\end{array}
\right.\quad
\left.\begin{array}{rl}
S_2\{1,2^{t-2},t^2\} =\hspace{-0.3cm} & [0, t\fr{-2} 2,
3\fr{+2} t+1, 1],\\
S_2\{1,2^{t-1},t^2\}=\hspace{-0.3cm} & [0, 2, t+2\fr{-2} 4, 3\fr{+2} t+1, 1].\\
\end{array}
\right.$$
\end{footnotesize}
It is not hard to check that all the above realizations are $\ext{t-4}$.
In fact in each case, the realization contains the edges
$[v-t+4x,v-t+2+4x]$
and $[v-t+1+4x,v-t+3+4x]$  for any $x=0,\ldots,\frac{t-4}{4}$.
Hence by Lemmas \ref{c,c+4w} and \ref{senzak}(i), the existence of the two first special
linear realizations
implies that of a special linear realization of
type $2$ of $\{1,2^{t-2}, t^c\}$ and of $\{1,2^{t-1}, t^c\}$, respectively, for any
positive $c \equiv 0 \pmod 4$. Also, the existence of the third and forth special linear
realizations implies that of
a special linear realization of
type $2$ of $\{1,2^{t-2}, t^c\}$ and of $\{1,2^{t-1}, t^c\}$, respectively, for any
$c\equiv 2 \pmod 4$.\\
Now from Lemma \ref{special2} we have that the existence of all these linear realizations
ensures that of a special linear realization of type $2$ of $\{1,2^b,t^c\}$
for any $b\geq t-2$ and any positive even integer $c$.
\medskip

\noindent Suppose now $t\equiv 2\pmod 4$, $t\geq 6$. We have the following special linear
realizations
of both type $1$ and $2$:
\begin{footnotesize}
$$\left.\begin{array}{rl}
S_{1,2}\{1,2^{t-2}\}=\hspace{-0.3cm} & [0\fr{+2} t-2, t-1\fr{-2} 1],\\
S_{1,2}\{1, 2^{t-1}\} =\hspace{-0.3cm} & [ 0\fr{+2} t, t-1
\fr{-2} 1],\\
        \end{array}\right.
\quad
\left.\begin{array}{rl}
S_{1,2}\{1,2^{t-1},t^{2}\}=\hspace{-0.3cm} &[0 , t\fr{-2} 2,
t+2, t+1\fr{-2} 1],\\
S_{1,2}\{1,2^t,t^{2}\}=\hspace{-0.3cm} & [ 0,t\fr{-2} 2
, t+2, t+3\fr{-2} 1 ].\\
        \end{array}\right.$$
  \end{footnotesize}
The reader can check that the above realizations are $\ext{t-6}$,
since they contain the edges  $[v-t+4x,v-t+2+4x]$
and $[v-t+1+4x,v-t+3+4x]$  for any $x=0,\ldots,\frac{t-6}{4}$.
Hence by Lemmas \ref{c,c+4w} and \ref{senzak}(ii)  we have the existence of a special
linear
realization of type $2$
for the following lists:  $\{1,2^{t-2}, t^{tk+d}\}$ for any $d\equiv 0\pmod4$;
$\{1,2^{t-1}, t^{c}\}$ for any $c$ even;
$\{1,2^{t}, t^{tk+d}\}$ for any $d\equiv2\pmod4$.
Thus, by Lemma \ref{special2} there exists a special linear realization of type $2$ of
$\{1,2^b,t^{tk+d}\}$
for any $b\geq t-2$ if $d\equiv0\pmod4$ and for any $b\geq t-1$ if $d\equiv2\pmod4$.
\end{proof}

\begin{prop}\label{t4a1}
Let $t\equiv 0 \pmod 4$, $t\geq4$.
The list $\{1, 2^b, t^c\}$ admits a linear realization for any $b\geq t-2$ and any
positive
odd integer $c$.
\end{prop}

\begin{proof}
If $t=4$ we have the following special linear realizations of type $2$:
\begin{footnotesize}
$$\left.\begin{array}{rl}
 S_2\{1,2^2,4^{4k+1}\}=\hspace{-0.3cm} &[0 \rightarrow 4k+4, 4k+2 \rightarrow 2, 3
\rightarrow 4k+3, 4k+1 \rightarrow 1 ],\\
S_2\{1,2^3,4^{4k+1}\}=\hspace{-0.3cm} &[0, 2 \rightarrow 4k+2, 4k+4 \rightarrow 4, 3
\rightarrow 4k+3, 4k+5 \rightarrow 1 ],\\
S_2\{1,2^2,4^{4k+3}\}=\hspace{-0.3cm} &[0 \rightarrow 4k+4, 4k+6 \rightarrow 2, 3
\rightarrow 4k+3, 4k+5 \rightarrow 1],\\
S_2\{1,2^3,4^{4k+3}\}=\hspace{-0.3cm} &[0, 2 \rightarrow 4k+6, 4k+4 \rightarrow 4, 3
\rightarrow 4k+7, 4k+5 \rightarrow 1 ].
\end{array}
\right.$$
\end{footnotesize}
Hence by Lemma \ref{special2} we have a linear realization of $L=\{1,2^b,4^c\}$ for any
$b\geq 2$ and any odd positive integer $c$.\\
So, we can assume $t\geq 8$. Write $c=tk+d$ with $k\geq 0$ and $0\leq d<t$.
 We split the proof into two parts according to the congruence class of
$d$  modulo $4$. \\
i) Let $d\equiv 1\pmod 4$.\\
For $b=t+1,t+2,2t-1, 2t$ we have
\begin{footnotesize}
$$\left.\begin{array}{rl}
S_1\{1,2^{t+1},t \}=\hspace{-0.3cm} &[0\fr{+2} t+2, t+3,
 t+1 ,1\fr{+2} t-1 ],\\
S_1\{1,2^{t+2},t\}=\hspace{-0.3cm} &[0\fr{+2} t+4, t+3, t+1,1\fr{+2}  t-1 ],\\
S_1\{1,2^{2t-1},t\}=\hspace{-0.3cm} &[ 0\fr{+2} 2t,  2t+1\fr{-2} t+1,1\fr{+2} t-1],\\
S_1\{1,2^{2t},t\}=\hspace{-0.3cm} &[ 0\fr{+2}  2t+2,
2t+1\fr{-2} t+1 , 1\fr{+2} t-1].\\
 \end{array}
\right.$$
\end{footnotesize}
Observe that these four special realizations are $\ext{t-4}$ and so we can apply Lemmas \ref{c,c+4w} and \ref{senzak}(i). Also, by
Lemma \ref{senzak}(iii), we are left to consider the following
cases.
For $b=t-2,t-1,t$ we have
\begin{footnotesize}
$$\left.\begin{array}{rl}
r\{1,2^{t-2},t\}=\hspace{-0.3cm} &[0, t\fr{-2} 2, 1\fr{+2}  t-1],\\
r\{1,2^{t-1},t\}=\hspace{-0.3cm} & [ 0\fr{+2} t, t-1, t+1,  1\fr{+2} t-3 ],\\
S_1\{1,2^t,t\}=\hspace{-0.3cm} &  [0\fr{+2} t+2, t+1, 1\fr{+2} t-1 ],\\
r\{1,2^t,t^{5}\}=\hspace{-0.3cm} &[0, 2 ,  t+2\fr{-2}  6,
t+6, t+4, 4, 3, 1, t+1\fr{-2}  5 , t+5, t+3
].
 \end{array}
\right.$$
\end{footnotesize}
\noindent For any $z$ such that $0\leq z\leq
\frac{t-12}{4}$ we
take
\begin{footnotesize}
$$\left.\begin{array}{rl}
r\{1,2^{t+4z+3},t^5\}=\hspace{-0.3cm} &[ 0\fr{+2} 4z+4, t+4z+4
\fr{-2} 4z+8, t+4z+8, t+4z+6 ,\\
&4z+6, 4z+7\fr{-2} 1, t+1 \fr{-2} 4z+9, t+4z+9 \fr{-2} t+3],\\
r\{1,2^{t+4z+4},t^5\}=\hspace{-0.3cm} &[ 0\fr{+2} 4z+6 ,t+4z+6\fr{-2} 4z+10,  t+4z+10,  t+4z+8, \\
&4z+8,  4z+7 \fr{-2} 1, t+1\fr{-2} 4z+9, t+4z+9\fr{-2} t+3 ].\\
 \end{array}
\right.$$
\end{footnotesize}
Now, for $b=2t-5,2t-4$ we have
\begin{footnotesize}
$$\left.\begin{array}{rl}
r\{1,2^{2t-5},t^{5}\}=\hspace{-0.3cm} &[ 0\fr{+2} t-4, 2t-4\fr{-2} t, 2t, 2t-2,  t-2, t-1\fr{-2} 1 \rightarrow 2t+1 \fr{-2}  t+3 ],\\
r\{1,2^{2t-4},t^{5}\}=\hspace{-0.3cm} &[ 0\fr{+2} t-2,  2t-2 \fr{-2} t+2, 2t+2, 2t, t, t-1\fr{-2} 1 \rightarrow 2t+1 \fr{-2} t+3 ].\\
 \end{array}
\right.$$
\end{footnotesize}
By an easy but long check one can see that all the previous realizations, except
$S_1\{1,2^t,t\}$ are $\ext{t-4}$.
Hence, we obtain the existence of a linear realization of $\{1,2^b,t^c\}$ for any $b\geq
t-2$ and any $c\equiv 1 \pmod4$.

\noindent ii) Let $d\equiv 3\pmod4$.\\
For $b=t+3, t+4, 2t+1, 2t+2$ we have
\begin{footnotesize}
$$\left.\begin{array}{rl}
S_1\{1,2^{t+3},t^3\}=\hspace{-0.3cm} & [ 0, 2, t+2\fr{-2} 4
, t+4, t+6, t+7\fr{-2} t+1, 1\fr{+2} t-1 ],\\
S_1\{1,2^{t+4},t^3\}=\hspace{-0.3cm} &[ 0, 2, 4 , t+4\fr{-2} 6, t+6, t+8, t+7\fr{-2} t+1, 1\fr{+2} t-1 ],\\
S_1\{1,2^{2t+1},t^3\}=\hspace{-0.3cm} &[ 0\fr{+2} 2t+4, 2t+5\fr{-2}  t+7 , 7\fr{+2} t+1 , 1, 3, 5, t+5, t+3 ],\\
S_1\{1,2^{2t+2},t^3\}=\hspace{-0.3cm} &[ 0\fr{+2} 2t+6,2t+5\fr{-2} t+7 , 7\fr{+2} t+1 , 1, 3,  5 , t+5, t+3 ].
 \end{array}
\right.$$
\end{footnotesize}
Observe that these four special realizations are $\ext{t-4}$ and so we can apply Lemmas \ref{c,c+4w} and \ref{senzak}(i). Also, by
Lemma \ref{senzak}(iii), we are left to consider the following
cases.

\noindent For $b=t-2,t-1,t,2t-3,2t-2$ we have
\begin{footnotesize}
$$\left.\begin{array}{rl}
r\{1,2^{t-2},t^3\}=\hspace{-0.3cm} &[0 ,t, t+2 , 2 \fr{+2} t-2,  t-1, t+1 , 1\fr{+2} t-3],\\
r\{1,2^{t-1},t^3\}=\hspace{-0.3cm} &[0, t, t+2 , 2\fr{+2} t-2, t-1\fr{-2} 1 , t+1, t+3],\\
r\{1,2^t,  t^3\}=\hspace{-0.3cm} &[0, 2 ,t+2, t+4 ,4\fr{+2} t, t-1\fr{-2} 1 , t+1, t+3 ],\\
r\{1,2^{2t-3},t^3\}=\hspace{-0.3cm} &[ 0\fr{+2} t , 2t\fr{-2} t+2, t+3\fr{+2} 2t+1 \rightarrow 1\fr{+2} t-1 ],\\
r\{1,2^{2t-2},t^3\}=\hspace{-0.3cm} &[ 0\fr{+2} t+2, 2t+2\fr{-2} t+4, t+3 \fr{+2}  2t+1 \rightarrow 1\fr{+2} t-1
].
 \end{array}
\right.$$
\end{footnotesize}
\noindent For any $y$ with $0\leq y\leq \frac{t-8}{4}$,
 we have the following linear realizations:
\begin{footnotesize}
$$\left.\begin{array}{rl}
r\{1,2^{t+4y+1},t^3\}=\hspace{-0.3cm} &[0\fr{+2} 4y+4 ,t+4y+4\fr{-2} 4y+6,4y+7\fr{+2} t+1, 1\fr{+2} 4y+5 ,t+4y+5\fr{-2} t+3],\\
r\{1,2^{t+4y+2},t^{3}\}=\hspace{-0.3cm} &[0\fr{+2} 4y+6, t+4y+6\fr{-2}  4y+8 ,4y+7\fr{+2} t+1, 1\fr{+2} 4y+5, t+4y+5\fr{-2} t+3 ].
 \end{array}
\right.$$
\end{footnotesize}
By an easy but long check one can see that  the last seven realizations
are $t^{t-4}$-extendable. So by
Lemma
\ref{senzak}(i)
we obtain the existence of a linear realization of $\{1,2^b,t^c\}$ for any $b\geq
t-2$ and any $c \equiv 3 \pmod4$.
\end{proof}

\begin{prop}\label{t2a1}
Let $t\equiv 2 \pmod 4$, $t\geq 6$ and let $c$ be an odd integer. Write $c=tk+d$, with
$k\geq 0$ and $0\leq d < t$. The list $\{1, 2^b, t^{tk+d}\}$ admits a linear realization
for any $b\geq t-1$ if $d\equiv 1 \pmod 4$; for any $b\geq t-2$ if $d\equiv 3\pmod 4$.
\end{prop}

\begin{proof}
First, consider the case $t=6$. We have the following linear realizations:
\begin{footnotesize}
$$\left.\begin{array}{rl}
r\{1,2^4,6^{6k+1}\}=\hspace{-0.3cm} &[0 \rightarrow 6k+6, 6k+4\rightarrow 4, 2\rightarrow
6k+2, 6k+1\rightarrow 1, 3\rightarrow 6k+3, 6k+5\rightarrow 5],\\
S_1\{1,2^4,6^{6k+3}\}=\hspace{-0.3cm} &[0 \rightarrow 6k+6, 6k+4 \rightarrow 4, 2
\rightarrow 6k+8, 6k+7 \rightarrow 1, 3 \rightarrow 6k+3, 6k+5 \rightarrow 5 ],\\
S_1\{1,2^4,6^{6k+5}\}=\hspace{-0.3cm} &[0 \rightarrow 6k+6, 6k+8 \rightarrow 2, 4
\rightarrow 6k+10, 6k+9 \rightarrow 3, 1 \rightarrow 6k+7, 6k+5 \rightarrow 5 ],\\
S_1\{1,2^5,6^{6k+1}\}=\hspace{-0.3cm} &[0, 2 \rightarrow 6k+2, 6k+4 \rightarrow 4, 6
\rightarrow 6k+6, 6k+7 \rightarrow 1, 3 \rightarrow 6k+3, 6k+5 \rightarrow 5
].
\end{array}
\right.$$
\end{footnotesize}
Hence by Lemma \ref{special1} we have the existence of a  linear realization
of
$L=\{1,2^b,6^c\}$ for any $b\geq 4$ and any odd integer $c$.
So from now on, we can assume $t\geq 10$. We split the proof into two parts according
to the congruence class of
$d$ modulo $4$. \\
\noindent i) Let $d\equiv 1\pmod 4$.\\
For $b=t-1,t,t+5,t+6$ we have the following special realizations of type $1$:
\begin{footnotesize}
$$\left.\begin{array}{rl}
S_1\{1,2^{t-1},t\}=\hspace{-0.3cm} &[ 0\fr{+2} t, t+1
, 1\fr{+2} t-1 ],\\
S_1\{1,2^t,t\}=\hspace{-0.3cm} &[ 0\fr{+2} t+2, t+1 ,1\fr{+2} t-1 ],\\
S_1\{1,2^{t+5},t^5\}=\hspace{-0.3cm} &[ 0\fr{+2} 8 , t+8\fr{-2}  10 , t+10, t+11, t+9, t+7
,7\fr{+2} t+1, 1, 3, 5 , t+5, t+3],\\
S_1\{1,2^{t+6},t^5\}=\hspace{-0.3cm} &[ 0\fr{+2} 10, t+10\fr{-2}  12 ,t+12, t+11, t+9 , t+7 , 7\fr{+2} t+1, 1, 3, 5 , t+5, t+3].\\
      \end{array}
\right.$$
\end{footnotesize}
Observe that the previous four special realizations of type $1$ are $\ext{t-6}$, so we can apply Lemmas \ref{c,c+4w} and \ref{senzak}(ii).
Furthermore, by Lemma \ref{special1} we can assume $c>1$ and by Lemma
\ref{senzak}(iii) we are left to consider
only linear realizations for $b=t+1,t+2$:
\begin{footnotesize}
$$\left.\begin{array}{rl}
r\{1,2^{t+1},t^{5}\}=\hspace{-0.3cm} &[ 0, 2, 4, t+4, t+6, 6\fr{+2}  t+2 , t+3, 3, 1, t+1 \fr{-2} 5, t+5, t+7],\\
r\{1,2^{t+2},t^{5}\}=\hspace{-0.3cm} &[ 0,2,4, 6, t+6, t+8, 8\fr{+2} t+4 , t+3, 3, 1 , t+1\fr{-2} 5 , t+5, t+7 ],\\
S_1\{1,2^{t+1},t^{t+1}\}=\hspace{-0.3cm} &  [ 0, 2, 4\rightarrow t+4, t+6\rightarrow 6,
8\rightarrow t+8,\ldots, t-2 \rightarrow 2t-2, 2t \rightarrow t, \\
& t+2 \rightarrow 2t+2, 2t+3 \rightarrow  3, 1\rightarrow 2t+1, 2t-1\rightarrow t-1,
t-3\rightarrow 2t-3, \\
& 2t-5\rightarrow t-5, \ldots,  t+9\rightarrow 9, 7, 5\rightarrow t+5, t+7 ],\\
S_1\{1,2^{t+2},t^{t+1}\}=\hspace{-0.3cm} &  [ 0, 2, 4,6\rightarrow t+6, t+8\rightarrow 8,
10\rightarrow t+10,\ldots, t \rightarrow 2t, 2t+2 \rightarrow t+2, \\
& t+4 \rightarrow 2t+4, 2t+3 \rightarrow  3, 1\rightarrow 2t+1, 2t-1\rightarrow t-1,
t-3\rightarrow 2t-3,\\
& 2t-5\rightarrow t-5, \ldots,  t+9\rightarrow 9, 7, 5\rightarrow t+5, t+7 ].\\
\end{array}
\right.$$
\end{footnotesize}
The first two linear realizations are $\ext{t-10}$, whereas the last two are  special of type $1$ and $\ext{t-6}$. Hence, we proved the existence of a linear realization of
$\{1,2^b,t^{tk+d}\}$
for any $b\geq t-1$ and $d\equiv 1 \pmod4$.

\noindent
ii) Let  $d\equiv 3 \pmod 4$.\\
For $b=t-2,t+1, t+3,t+4$ we have the following special linear realizations of type $1$:
\begin{footnotesize}
$$\left.\begin{array}{rl}
S_1\{1,2^{t-2},t^{3}\} =\hspace{-0.3cm} & [0 ,t \fr{-2} 2 , t+2, t+1, 1\fr{+2}
t-1],\\
S_1\{1,2^{t+1},t^3\}=\hspace{-0.3cm} & [0,2 , t+2\fr{-2} 4, t+4, t+5 , t+3 , t+1, 1\fr{+2}
 t-1],\\
S_1\{1,2^{t+3},t^3\}=\hspace{-0.3cm} &[ 0\fr{+2} t+6, t+7, 7\fr{+2} t+1, 1, 3, 5, t+5, t+3],\\
S_1\{1,2^{t+4},t^3\}=\hspace{-0.3cm} &[ 0\fr{+2} t+8, t+7, 7\fr{+2} t+1, 1, 3,5, t+5, t+3].\\
\end{array}
\right.$$
\end{footnotesize}
Observe that  these four linear realizations are
$\ext{t-6}$. By Lemmas \ref{c,c+4w} and
\ref{senzak}(iii), we are left to consider the following linear realizations for
$b=t-1,t$:
\begin{footnotesize}
$$\left.\begin{array}{rl}
r\{1,2^{t-1},t^{tk+3}\}=\hspace{-0.3cm} & [0\rightarrow tk+t, tk+t+2\rightarrow2,
4\rightarrow tk+4, tk+6\rightarrow6,\ldots, t-2\rightarrow tk+t-2,\\ & tk+t-1\rightarrow
t-1, t-3\rightarrow tk+t-3, \ldots, tk+5\rightarrow5, 3, 1\rightarrow tk+t+1,\\
&
tk+t+3\rightarrow t+3 ],\\
r\{1,2^t,t^{tk+3}\}=\hspace{-0.3cm} & [0, 2\rightarrow tk+t+2, tk+t+4\rightarrow4,
6\rightarrow tk+6, tk+8\rightarrow8,\ldots, t\rightarrow tk+t,\\ & tk+t-1\rightarrow t-1,
t-3\rightarrow tk+t-3, \ldots, tk+5\rightarrow5, 3, 1\rightarrow tk+t+1,\\
& tk+t+3\rightarrow t+3 ].\\
\end{array}
\right.$$
\end{footnotesize}
Observe that the two previous realizations are
$\ext{t-10}$. Hence, by Lemma \ref{c,c+4w} we obtain a linear realization of
$\{1,2^b,t^{tk+d}\}$ for any
$b\geq
t-2$ and any $d\equiv 3 \pmod 4$.
\end{proof}

\begin{rem}\label{cyclict2a1}
Note that in previous propositions we have not constructed a linear realization of
$\{1,2^{t-2},t^{tk+d}\}$ when $t\equiv 2 \pmod 4$ and $d\equiv 1,2 \pmod4$.
However we point out that there exists a cyclic realization of both
$\{1,2^{t-2},t^{tk+t+4x+1}\}$ and $\{1,2^{t-2},t^{tk+t+4x+2}\}$, with $0<4x+1,4x+2<t$. Namely, we have
\begin{footnotesize}
$$\left.\begin{array}{rl}
r\{1,2^{t-2},t^{tk+1}\}=\hspace{-0.3cm} &[0 \rightarrow tk+t, tk+t-2\rightarrow
t-2,\ldots, tk+4 \rightarrow 4, 2\rightarrow tk+2,\\
&tk+1 \rightarrow 1, 3\rightarrow tk+3, \ldots, tk+t-1\rightarrow t-1],\\

c\{1^1,2^{t-2},t^{tk+t+4x+5}\}=\hspace{-0.3cm} &
[ 0\rightarrow tk+2t, tk+2t+2\rightarrow 2, 4\rightarrow
tk+2t+4, \ldots, tk+2t+4x+2\rightarrow\\
& 4x+2,
4x+4\rightarrow tk+2t+4x+4,
1\rightarrow tk+2t+1, tk+2t+3\rightarrow 3,\\
& \ldots,4x+1\rightarrow tk+2t+4x+1,
tk+2t+4x+3\rightarrow 4x+3, 4x+5\rightarrow \\
& tk+t+4x+5,
tk+t+4x+7\rightarrow 4x+7, \ldots, t-1\rightarrow tk+t-1,\\
& tk+t-2\rightarrow t-2, t-4\rightarrow tk+t-4, \ldots, 4x+6\rightarrow tk+t+4x+6],\\

S_2\{1,2^{t-2},t^{tk+2}\}=\hspace{-0.3cm} &[0 \rightarrow tk+t, tk+t-2\rightarrow
t-2,\ldots, tk+4 \rightarrow 4, 2\rightarrow tk+2, \\
&tk+3 \rightarrow 3,
5\rightarrow tk+5, \ldots, t-1\rightarrow tk+t-1, tk+t+1\rightarrow 1],\\

c\{1,2^{t-2},t^{tk+t+4x+6}\}=\hspace{-0.3cm} &[ 0, 2\rightarrow tk+2t+2,
tk+2t+4\rightarrow 4, 6\rightarrow tk+2t+6, tk+2t+8\rightarrow 8,\\
& \ldots, 4x+2\rightarrow
tk+2t+4x+2, tk+2t+4x+4\rightarrow 4x+4, 4x+6\rightarrow \\
& tk+t+4x+6, tk+t+4x+8\rightarrow
4x+8, \ldots, t\rightarrow tk+2t, tk+2t+1\\
& \rightarrow 1, 3\rightarrow tk+2t+3,
tk+2t+5\rightarrow 5, 7\rightarrow tk+2t+7, \ldots, 4x+3\rightarrow \\
&tk+2t+4x+3, t-3,
t-5\rightarrow tk+2t-5, tk+2t-7\rightarrow t-7, \ldots, \\
& 4x+5\rightarrow tk+2t+4x+5,
t-1\rightarrow tk+2t-1, tk+2t-3\rightarrow 2t-3].
 \end{array}
\right.$$
\end{footnotesize}
\end{rem}

\section{Construction of linear realizations of $\{1^a,2^b,t^c\}$ for $a\geq 2$}\label{4}

In this section we investigate $\BHR(\{1^a,2^b,t^c\})$ when $a\geq 2$ and $a+b\geq t-1$. In view of Remark
\ref{RemA>a}, we are left to consider only the cases when $a+b=t-1$.

\begin{prop}\label{a+b=t-1}
There exists a linear realization of $L=\{1^a,2^b,t^c\}$ for any even
integer $t\geq 4$ and any positive integer $a,b,c$ with $a\geq2$ and $a+b=t-1$.
\end{prop}

\begin{proof}
Since by hypothesis  $a+b=t-1$ and, obviously, we have $t$ congruence classes modulo $t$,
we use $\pm 1$ or $\pm 2$ to change
the congruence class. As said above, all our linear realizations start from $0$, so  adding $t$
a suitable number of times, we take
all the elements of the congruence class $0$
modulo $t$. Then applying $\pm 1$ or $\pm 2$ we go into another congruence class modulo
$t$ and we take all the elements of this class adding or subtracting $t$. Next, we
change again class applying $\pm 1$ or $\pm 2$
and so on. In this way we have to reach all the congruence classes modulo $t$, namely all
the elements from $0$ to $|L|$,
using exactly $a$ times $1$, $b$ times $2$ and $c$ times $t$.\\
In order to obtain a complete solution constructing
a few number of realizations we look for those which are $t^{4w}$-extendable.
In particular, all the linear realizations of $\{1^a,2^b,t^c\}$  we are going to construct  with $a+b=t-1$
and $c=0,1,2,3$ are $t^{4w}$-extendable with
 $w=\lfloor\frac{t}{4}\rfloor-1$. Then applying Lemmas \ref{c,c+4w} and \ref{senzak}
we have a linear realization for any list  $\{1^a,2^b,t^c\}$ with $a+b=t-1$ and $c\geq1$.

\medskip
\noindent \emph{Case 1.} Let $a \equiv 0 \pmod4$.\\
We start with the case $t\equiv0\pmod4$, set $t=4u\geq8$.
So we consider the list
$L=\{1^{4y+4},2^{t-4y-5},t^c\}$
with  $0\leq y\leq u-2$ (since $1\leq a \leq t-2$).

\noindent A linear realization of $L$ for $c=0$ and $c=2$ is, respectively
$$P_1=[0\stackrel{+1}{\rightharpoondown}4y+3\stackrel{+2}{\rightharpoondown}4u-1,4u-2\stackrel{-2}{\rightharpoondown}4y+4]$$
and
$$P_2=[0,t,t+1,1\stackrel{+1}{\rightharpoondown}4y+1,4y+3,4y+2,4y+4,4y+5\stackrel{+2}{\rightharpoondown}4u-1,4u-2\stackrel{-2}{\rightharpoondown}4y+6].$$

Furthermore, starting from the realization $P_1$
we can obtain a linear realization for $c=1$ by changing the sign of each integer $+1,-1,+2,-2$
used to change the congruence classes modulo $t$; observe that now $|L|=t$.
Also, starting from the realization $P_2$  we obtain a linear
realization for $c=3$, again  by changing the sign of each integer
$+1,-1,+2,-2$; note that now $|L|=t+2$.

Suppose now $t\equiv 2\pmod 4$; set $t=4u-2\geq 6$.
Starting from the realization constructed above for $L$, when $t=4u\geq8$,
we obtain a realization for $\{1^{4y+4},2^{t-4y-5},t^{c}\}$ when $t=4u-2$.
If $c=0$, it is sufficient to replace the subpath
$[4u-3,4u-1,4u-2,4u-4]$ of $P_1$ with the edge  $[4u-3,4u-4]$.

Consider $c=2$. If $y\leq u-3$, then replace in $P_2$
the subpath
$[4u-3,4u-1,4u-2,4u-4]$  with the edge  $[4u-3,4u-4]$.

If $y=u-2$, i.e. $L=\{1^{t-2},2,t^{2}\}$, we have the following linear realization:
$$\left.\begin{array}{rl}
S_1 L=\hspace{-0.3cm} & [0, t,  t+1, 1\stackrel{+1}{\rightharpoondown}  t-3,
 t-1, t-2].
\end{array}
\right.$$

If $c=1$, the realization can be obtained as done for the case $t\equiv
0\pmod 4$, namely it is sufficient to change the signs of the
all $1$'s and $2$'s used to construct the realization when $c=0$.
Finally, if $c=3$ a linear realization is
$$P_3=[0,t \stackrel{-1}{\rightharpoondown}t-4y-3\stackrel{-2}{\rightharpoondown}1,t+1,t+2,2\stackrel{+2}{\rightharpoondown}t-4y-4].$$

At this point we have obtained a linear realization for
$\{1^{4y+4},2^{t-4y-5},t^c\}$,  for all even $t\geq 8$ and any positive integer $c$.

\medskip
\noindent \emph{Case 2.} Let $a \equiv 2 \pmod 4$.\\
We start with the case $t\equiv0\pmod4$, set $t=4u\geq 4$.
So we consider the list
$L=\{1^{4y+2},2^{t-4y-3},t^c\}$
with  $0\leq y\leq u-1$ (since $1\leq a \leq t-2$).

If $t\geq 8$ and $0\leq y\leq u-2$, to obtain a linear realization of
$L=\{1^{4y+2},2^{t-4y-3}\}$, replace in $P_1$
the subpath
$[4y\stackrel{+1}{\rightharpoondown}4y+3]$  with  $[4y,4y+2,4y+1,4y+3]$.

If $t\geq 4 $ and $y=u-1$ (i.e. $L=\{1^{t-2},2\}$) we have
$$\left.\begin{array}{rl}
S_1 L=\hspace{-0.3cm} & [0\stackrel{+1}{\rightharpoondown} t-3, t-1, t-2].
\end{array}
\right.$$

Applying the reasoning explained in Case 1, we have also a linear realization of  $\{1^{4y+2},2^{t-4y-3},t\}$.
Now, a linear realization of $L=\{1^{4y+2},2^{t-4y-3},t^{2}\}$ is
$$P_4=[0,t,t+1,1\stackrel{+1}{\rightharpoondown}4y+1\stackrel{+2}{\rightharpoondown}4u-1,4u-2\stackrel{-2}{\rightharpoondown}4y+2]$$
and from it
we can also obtain a linear realization of  $\{1^{4y+2},2^{t-4y-3},t^{3}\}$,
again changing the signs of $1$'s and $2$'s.

As in Case 1, to solve the case $t=4u-2\geq 6$, we start from the realizations
of $L$ when $t=4u\geq8$. In fact, in order to obtain a linear realization of
$\{1^{4y+2},2^{t-4y-3}\}$ when  $t=4u-2\geq 6$, it is sufficient to replace in
$r\{1^{4y+2},2^{t-4y-3}\}$ constructed for $t=4u$, the
subpath $[4u-3,4u-1,4u-2,4u-4]$ with $[4u-3,4u-4]$.

Analogously, in order to obtain a linear realization of
$\{1^{4y+2},2^{t-4y-3},t^{2}\}$ when  $t=4u-2\geq 6$, it is sufficient to replace in
$P_4$ the subpath $[4u-3,4u-1,4u-2,4u-4]$ with $[4u-3,4u-4]$.

As before, if $c=1$  a $rL$ can be obtained changing the signs
of all the $1$'s and $2$'s used to obtain the realization in the case $c=0$.
Finally, if $c=3$, we replace in $P_3$ the subpath $[t\stackrel{-1}{\rightharpoondown}t-3]$
with $[t,t-2,t-1,t-3]$.

\medskip
\noindent \emph{Case 3.} Let $a \equiv 3 \pmod 4$.\\
We start with the case $t\equiv0\pmod4$, set $t=4u\geq 8$.
So we consider the list
$L=\{1^{4y+3},2^{t-4y-4},t^c\}$
with  $0\leq y\leq u-2$ (since $1\leq a \leq t-2$).
A linear realization for $c=0$
 and for $c=2$ is, respectively,
 $$P_5=[0\stackrel{+1}{\rightharpoondown}4y+1,4y+3,4y+2\stackrel{+2}{\rightharpoondown}4u-2,4u-1\stackrel{-2}{\rightharpoondown}4y+5]$$
and
$$P_6=[0,t,t+1,1\stackrel{+1}{\rightharpoondown}4y+2\stackrel{+2}{\rightharpoondown}4u-2,4u-1\stackrel{-2}{\rightharpoondown}4y+3].$$

The reader can check that, reasoning as in Case 1, it is possible to construct
a linear realization of $L$ for $c=1$ and $c=3$ starting from the realization
$P_5$ and $P_6$, respectively.

Also here, to solve the case $t=4u-2\geq6$, we start from the realizations
of $L$ when $t=4u\geq 8$. In fact, in order to obtain a linear realization of
$\{1^{4y+3},2^{t-4y-4}\}$ when  $t=4u-2\geq 6$, it is sufficient to replace in
$P_5$ the subpath $[4u-4,4u-2,4u-1,4u-3]$ with $[4u-4,4u-3]$.

Analogously, in order to obtain a linear realization of
$\{1^{4y+3},2^{t-4y-4},t^{2}\}$ when  $t=4u-2\geq 6$, it is sufficient to replace in
$P_6$ the subpath $[4u-4,4u-2,4u-1,4u-3]$ with $[4u-4,4u-3]$.

As before, if $c=1$  a $rL$ can be obtained changing the signs
of all $1$'s and $2$'s used in
 the case $c=0$. Finally, if $c=3$, a linear
 realization of $L$ is
 $$P_7=[0,t,t-1,t-3,t-2,t-4\stackrel{-1}{\rightharpoondown}t-4y-4\stackrel{-2}{\rightharpoondown} 2,t+2,t+1,1\stackrel{+2}{\rightharpoondown}t-4y-5].$$

\medskip
\noindent \emph{Case 4.} Let $a \equiv 1 \pmod4$.\\
We start again with the case $t\equiv0\pmod4$, set $t=4u\geq 8$.
Since $a\geq2$, we can assume $a=4y+5$ with $0\leq
y\leq u-2$.\\
From the linear realization $P_5$ of $\{1^{4y+3},2^{t-4y-4}\}$
we can obtain a linear
realization of $\{1^{4y+5},2^{t-4y-6}\}$. It suffices to replace
in $P_5$ the subpath $[4y+1,4y+3,4y+2,4y+4]$ with $[4y+1\stackrel{+1}{\rightharpoondown}4y+4]$.

Moreover, if $c=2$, a linear
realization of $\{1^{4y+5},2^{t-4y-6},t^{2}\}$
is
$$P_8=[0,t,t+1,1\stackrel{+1}{\rightharpoondown}4y+2,4y+4,4y+3,4y+5,4y+6\stackrel{+2}{\rightharpoondown}4u-2,4u-1\stackrel{-2}{\rightharpoondown}4y+7].$$
Also in this case the reader can check that it is possible
to obtain a linear realization of $\{1^{4y+5},2^{t-4y-6},t^{c}\}$
for $c=1$ and $c=3$ changing the signs of the 1's and 2's used to construct the realization
when $c=0$ and $c=2$, respectively.

Also here, to solve the case $t=4u-2\geq 10$, we start from the realizations
of $L$ when $t=4u\geq12$. In fact, in order to obtain a linear realization of
$\{1^{4y+5},2^{t-4y-6}\}$ when  $t=4u-2\geq10$, it suffices to replace in $r\{1^{4y+5},2^{t-4y-6}\}$ constructed for $t=4u$ the subpath $[4u-4,4u-2,4u-1,4u-3]$ with $[4u-4,4u-3]$.

Analogously, in order to obtain a linear realization of
$\{1^{4y+5},2^{t-4y-6},t^{2}\}$ when  $t=4u-2\geq 10$, it suffices to replace in
$P_8$ the subpath $[4u-4,4u-2,4u-1,4u-3]$ with $[4u-4,4u-3]$.

As before, if $c=1$ a $rL$ can be obtained changing the signs
of all $1$'s and $2$'s
used in the case $c=0$. Finally, if $c=3$, it suffices to replace in
$P_7$ the subpath $[t-1,t-3,t-2,t-4]$ with $[t-1\stackrel{-1}{\rightharpoondown}t-4]$.
\end{proof}

Now we can prove the main result of this paper, that is Theorem \ref{main}.

\begin{proof}[Proof of Theorem \emph{\ref{main}}]
We prove the theorem using the linear realizations constructed in Sections \ref{3}
and \ref{4}. Observe that since $t$ appears in $L=\{1^a,2^b,t^c\}$, we have $|L|\geq 2t-1$
and so we
can apply Remark \ref{cyclin}.

\smallskip
Let $t\equiv 0 \pmod 4$. Combining Propositions \ref{c even} and  \ref{t4a1} with Remark
\ref{RemA>a}, we obtain a linear realization for every list $L$ with
$a,c\geq 1$ and $b\geq t-2$.
If $1\leq b\leq t-3$, then $a\geq2$ and hence the thesis follows from Proposition
\ref{a+b=t-1}
and Remark \ref{RemA>a}.\\
Now, let $t\equiv 2 \pmod 4$. Write $c=tk+d$ with $k\geq 0$ and $0\leq d< t$.
Combining Propositions \ref{c even} and  \ref{t2a1} with
Remark
\ref{RemA>a}, we obtain a linear realization for every list $\{1^a,2^b,t^{tk+d}\}$ with
$a\geq 1$, $b\geq t-2$, $k\geq 0$  and  $d\equiv 0,3\pmod 4$.

\smallskip
Assume $d\equiv 1,2\pmod4$.
If $b\geq t-1$ proceed as
before. When $a=1$ and $b=t-2$ we have a cyclic realization, see Remark \ref{cyclict2a1}.
If $a=2$, we have the following $t^{t-6}$-extendable linear realizations:
$$\left.\begin{array}{rl}
S_1\{1^2,2^{t-2},t\}=\hspace{-0.3cm} & [0, t, t+1 \stackrel{-2}{\rightharpoondown} 1, 2 \stackrel{+2}{\rightharpoondown} t-2],\\
S_1\{1^2,2^{t-2},t^{2}\}=\hspace{-0.3cm} & [0,2, t+2, t+1,
1\stackrel{+2}{\rightharpoondown} t-1, t \stackrel{-2}{\rightharpoondown} 4].
\end{array}
\right.$$
Hence by Lemmas \ref{c,c+4w} and \ref{senzak}(ii) we have a linear realization of
$\{1^2,2^{t-2},t^{tk+d}\}$ and applying Remark \ref{RemA>a} we have a linear
realization of
$\{1^a,2^{t-2},t^{tk+d}\}$
for any $a\geq2$.
Also in this case, if $1\leq b\leq t-3$, then $a\geq2$ and hence the thesis follows from
Proposition \ref{a+b=t-1}
and Remark \ref{RemA>a}.
\end{proof}

\section{A complete solution for $t=4,6,8$}

In view of Theorem \ref{main}, in order to prove $\BHR(\{1^a,2^b,t^c \})$ for
all $a,b,c\geq 1$ and $t$ even, we are left to construct cyclic realizations only when
$a+b\leq t-2$. As examples, we prove $\BHR(\{1^a,2^b,4^c \})$, $\BHR(\{1^a,2^b,6^c
 \})$ and $\BHR(\{1^a,2^b,8^c
 \})$.

\begin{prop}
$\BHR(\{1^a, 2^b, 4^c\})$ holds for all $a,b,c\geq 1$.
\end{prop}

\begin{proof}
Firstly, observe that since $4$ appears as edge-length, then $v=a+b+c+1\geq 8$.
Furthermore, the assumptions $a,b,c\geq 1$ reduce the conjecture to prove that $L=\{1^a,
2^b, 4^c\}$ admits a cyclic realization except when $4$ divides $v$ and $a+b=2$.
If $a+b\geq 3$, Remark \ref{cyclin} and Theorem \ref{main}  imply the
existence of such  a realization of $L$. For
$a+b=2$ (i.e., $a=b=1$) we have
\begin{footnotesize}
$$\left.\begin{array}{rl}
c\{1,2,4^{4k+6}\}=\hspace{-0.3cm} & [ 0\rightarrow 4k+8, 3\rightarrow 4k+7,
4k+5\rightarrow 1, 2\rightarrow 4k+6 ], \\
c\{1,2,4^{4k+7}\}=\hspace{-0.3cm} & [ 0\rightarrow 4k+8, 2\rightarrow 4k+6,
4k+7\rightarrow 3, 1\rightarrow 4k+9 ], \\
c\{1,2,4^{4k+8}\}=\hspace{-0.3cm} & [ 0\rightarrow 4k+8, 1\rightarrow 4k+9,
4k+7\rightarrow 3, 2\rightarrow 4k+10 ]. \\
 \end{array}
\right.$$
\end{footnotesize}
\end{proof}

\begin{prop}
$\BHR(\{1^a, 2^b, 6^c\})$ holds for all $a,b,c\geq 1$.
\end{prop}

\begin{proof}
Firstly, observe that since $6$ appears as edge-length, then $v=a+b+c+1\geq 12$.
Furthermore, the assumptions $a,b,c\geq 1$ reduce the conjecture to prove that $L=\{1^a,
2^b, 6^c\}$ admits a cyclic realization except when $6$ divides $v$ and
$a+b\leq 4$.
If $a+b\geq 5$, Remark \ref{cyclin} and Theorem \ref{main} imply the existence of such a
realization of $L$. For
$a+b\leq 4$  we have
\begin{footnotesize}
$$\left.\begin{array}{rl}
c\{1,2,6^{6k+10}\}=\hspace{-0.3cm} &[0\rightarrow 6k+12, 5\rightarrow 6k+11,
4\rightarrow 6k+10, 3\rightarrow 6k+9, 6k+7\rightarrow 1, 2\rightarrow 6k+8],\\
c\{1,2,6^{6k+11}\}=\hspace{-0.3cm} & [0\rightarrow 6k+12,  4\rightarrow 6k+10,
2\rightarrow 6k+8, 6k+9\rightarrow 3, 1\rightarrow 6k+13, 5\rightarrow 6k+11],\\
c\{1,2,6^{6k+12}\}=\hspace{-0.3cm} &[0\rightarrow 6k+12, 3\rightarrow 6k+9,
6k+10\rightarrow 4, 6k+13\rightarrow 1, 6k+14\rightarrow 2, 6k+11\rightarrow 5],\\
c\{1,2,6^{6k+13}\}=\hspace{-0.3cm} &[0\rightarrow 6k+12, 2\rightarrow 6k+14, 4\rightarrow
6k+10, 6k+11\rightarrow 5, 6k+15\rightarrow 3, 1\rightarrow 6k+13],\\
c\{1,2,6^{6k+14}\}=\hspace{-0.3cm} & [ 0\rightarrow 6k+12, 1\rightarrow 6k+13,
6k+11\rightarrow 5, 6k+16\rightarrow 4, 6k+15\rightarrow 3, 2\rightarrow 6k+14 ],\\[4pt]

c\{1,2^2,6^{6k+9}\}=\hspace{-0.3cm} &  [0\rightarrow 6k+12, 5\rightarrow 6k+11,
4\rightarrow 6k+10, 6k+8\rightarrow 2, 1\rightarrow 6k+7, 6k+9\rightarrow 3],\\
c\{1,2^2,6^{6k+10}\} =\hspace{-0.3cm} & [0\rightarrow 6k+12, 4\rightarrow 6k+10,
2\rightarrow 6k+8, 6k+9\rightarrow 3, 1\rightarrow 6k+13, 6k+11\rightarrow 5 ],\\
c\{1,2^2,6^{6k+11}\}  =\hspace{-0.3cm} &[0\rightarrow 6k+12, 3\rightarrow 6k+9,
6k+10\rightarrow 4, 2\rightarrow 6k+14, 5\rightarrow 6k+11, 6k+13\rightarrow 1 ],\\
c\{1,2^2,6^{6k+12}\} =\hspace{-0.3cm} &[0\rightarrow 6k+12,  2\rightarrow 6k+14,
4\rightarrow 6k+10, 6k+11\rightarrow 5, 3\rightarrow 6k+15,  1\rightarrow 6k+13],\\
c\{1,2^2,6^{6k+13}\} =\hspace{-0.3cm} &[0\rightarrow 6k+12, 1\rightarrow 6k+13,
6k+11\rightarrow 5, 6k+16\rightarrow 4, 2\rightarrow 6k+14, 6k+15\rightarrow 3 ],\\[4pt]

c\{1,2^3,6^{6k+8}\}=\hspace{-0.3cm} &[0\rightarrow 6k+12, 5\rightarrow 6k+11,
6k+9\rightarrow 3, 1\rightarrow 6k+7, 6k+8\rightarrow 2,4\rightarrow 6k+10 ],\\
c\{1,2^3,6^{6k+9}\}=\hspace{-0.3cm} &[0\rightarrow 6k+12, 4\rightarrow 6k+10,
6k+8\rightarrow 2, 1\rightarrow 6k+13, 6k+11\rightarrow 5, 3\rightarrow 6k+9 ],\\
c\{1,2^3,6^{6k+10}\}=\hspace{-0.3cm} &[0\rightarrow 6k+12, 3\rightarrow 6k+9,
6k+10\rightarrow 4, 2\rightarrow 6k+14, 1\rightarrow 6k+13, 6k+11\rightarrow 5 ],\\
c\{1,2^3,6^{6k+11}\}=\hspace{-0.3cm} &[0\rightarrow 6k+12, 2\rightarrow 6k+14,
4\rightarrow 6k+10, 6k+11\rightarrow 5, 7\rightarrow 6k+13, 6k+15\rightarrow 3, 1 ],\\
c\{1,2^3,6^{6k+12}\}=\hspace{-0.3cm} &[0\rightarrow 6k+12, 1\rightarrow 6k+13,
6k+11\rightarrow 5, 3\rightarrow 6k+15, 6k+14\rightarrow 2, 4\rightarrow 6k+16 ],\\[4pt]

c\{1^2,2,6^{6k+9}\} =\hspace{-0.3cm} &[0\rightarrow 6k+12,  5\rightarrow 6k+11,
4\rightarrow 6k+10, 6k+9\rightarrow 3, 1\rightarrow 6k+7, 6k+8\rightarrow 2],\\
c\{1^2,2,6^{6k+10}\}=\hspace{-0.3cm} &[0\rightarrow 6k+12, 4\rightarrow 6k+10,
6k+11\rightarrow 5, 6k+13\rightarrow 1,  3\rightarrow 6k+9, 6k+8\rightarrow 2 ],\\
c\{1^2,2,6^{6k+11}\}=\hspace{-0.3cm} &[0\rightarrow 6k+12, 3\rightarrow 6k+9,
6k+11\rightarrow 5, 4\rightarrow 6k+10, 1\rightarrow 6k+13, 16k+14\rightarrow 2],\\
c\{1^2,2,6^{6k+12}\}=\hspace{-0.3cm} &[0\rightarrow 6k+12, 2\rightarrow 6k+14,
6k+13\rightarrow 1, 3\rightarrow 6k+15, 5\rightarrow 6k+11, 6k+10\rightarrow 4],\\
c\{1^2,2,6^{6k+13}\}=\hspace{-0.3cm} &[0\rightarrow 6k+12, 1\rightarrow 6k+13,
6k+11\rightarrow 5, 4\rightarrow 6k+16, 6k+15\rightarrow 3, 6k+14\rightarrow 2 ],\\[4pt]

c\{1^2,2^2,6^{6k+8}\}=\hspace{-0.3cm} &[0\rightarrow 6k+12,  5\rightarrow 6k+11,
6k+9\rightarrow 3, 4\rightarrow 6k+10, 6k+8\rightarrow 2, 1\rightarrow 6k+7 ],\\
c\{1^2,2^2,6^{6k+9}\}=\hspace{-0.3cm} &[0\rightarrow 6k+12, 4\rightarrow 6k+10,
6k+9\rightarrow 3, 5\rightarrow 6k+11, 6k+13\rightarrow 1, 2\rightarrow 6k+8],\\
c\{1^2,2^2,6^{6k+10}\}=\hspace{-0.3cm} &[0\rightarrow 6k+12, 3\rightarrow 6k+9,
6k+10\rightarrow 4, 5\rightarrow 6k+11, 6k+13\rightarrow 1, 6k+14\rightarrow 2],\\
%
c\{1^2,2^2,6^{6k+11}\}=\hspace{-0.3cm} & [0\rightarrow 6k+12, 2\rightarrow 6k+14,
6k+15\rightarrow 3, 1\rightarrow 6k+13, 6k+11\rightarrow 5, 4\rightarrow 6k+10 ],\\
c\{1^2,2^2,6^{6k+12}\}=\hspace{-0.3cm} &[0\rightarrow 6k+12,  1\rightarrow 6k+13,
6k+14\rightarrow 2, 4\rightarrow 6k+16, 6k+15\rightarrow 3, 5\rightarrow 6k+11 ],\\[4pt]

c\{1^3,2,6^{6k+8}\}=\hspace{-0.3cm} &[0\rightarrow 6k+12, 5\rightarrow 6k+11,
6k+10\rightarrow 4, 3\rightarrow 6k+9, 6k+7\rightarrow 1, 2\rightarrow 6k+8 ],\\
c\{1^3,2,6^{ 6k+9}\}=\hspace{-0.3cm} & [0\rightarrow 6k+12, 4\rightarrow 6k+10,
6k+11\rightarrow 5, 3\rightarrow 6k+9, 6k+8\rightarrow 2, 1\rightarrow 6k+13],\\
c\{1^3,2,6^{6k+10}\}=\hspace{-0.3cm} &[0\rightarrow 6k+12, 3\rightarrow 6k+9,
6k+10\rightarrow 4, 5\rightarrow 6k+11, 6k+13\rightarrow 1, 2\rightarrow 6k+14 ],\\
c\{1^3,2,6^{6k+11}\}=\hspace{-0.3cm} &[0\rightarrow 6k+12,  2\rightarrow 6k+14,
6k+13\rightarrow 1, 6k+15\rightarrow 3, 4\rightarrow 6k+10, 6k+11\rightarrow 5 ],\\

c\{1^3,2,6^{6k+12}\}=\hspace{-0.3cm} &[0\rightarrow 6k+12, 1\rightarrow 6k+13,
6k+11\rightarrow 5, 4\rightarrow 6k+16, 6k+15\rightarrow 3, 2\rightarrow 6k+14 ].\\
\end{array}
\right.$$
\end{footnotesize}
\end{proof}

\begin{prop}
$\BHR(\{1^a, 2^b, 8^c\})$ holds for all $a,b,c\geq 1$.
\end{prop}

\begin{proof}
Firstly, observe that since $8$ appears as edge-length, then $v=a+b+c+1\geq 16$.
Furthermore, the assumptions $a,b,c\geq 1$ reduce the conjecture to prove that $L=\{1^a,
2^b, 8^c\}$ admits a cyclic realization except when either $4$ divides $v$ and
$a=b=1$ or  $8$ divides $v$ and $a+b\leq 6$.
If $a+b\geq 7$, Remark \ref{cyclin} and Theorem \ref{main} imply the existence of  such a
realization of $L$. For
$a+b\leq 6$  we have

\begin{footnotesize}
$$\left.\begin{array}{rl}
c\{1,2,8^{8k+14}\}=\hspace{-0.3cm} &[0\rightarrow 8k+16, 7\rightarrow 8k+15, 6\rightarrow
8k+14, 5\rightarrow 8k+13, 4\rightarrow 8k+12, 3\rightarrow 8k+11, \\
& 8k+9\rightarrow 1, 2\rightarrow 8k+10 ],\\

c\{1,2,8^{8k+15}\}=\hspace{-0.3cm} &[0\rightarrow 8k+16, 6\rightarrow 8k+14, 4\rightarrow
8k+12, 2\rightarrow 8k+10, 8k+11\rightarrow 3, 1\rightarrow 8k+17,\\
& 7\rightarrow 8k+15,
5\rightarrow 8k+13 ],\\

c\{1,2,8^{8k+16}\}=\hspace{-0.3cm} &[0\rightarrow 8k+16, 5\rightarrow 8k+13, 2\rightarrow
8k+18, 1\rightarrow 8k+17, 6\rightarrow 8k+14, 3\rightarrow 8k+11,\\
& 8k+12\rightarrow 4, 8k+15\rightarrow 7 ],\\

c\{1,2,8^{8k+18}\}=\hspace{-0.3cm} & [0\rightarrow 8k+16,  3\rightarrow 8k+19,
6\rightarrow 8k+14, 1\rightarrow 8k+17, 4\rightarrow 8k+12, 8k+13\rightarrow 5,  \\
& 8k+18, 8k+20, 7\rightarrow 8k+15, 2\rightarrow 8k+10 ],\\

c\{1,2,8^{8k+19}\}=\hspace{-0.3cm} &[0\rightarrow 8k+16, 2\rightarrow 8k+18, 4\rightarrow
8k+20, 6\rightarrow 8k+14, 8k+15\rightarrow 7, 8k+21\rightarrow 5,\\
& 8k+19\rightarrow 3,
1\rightarrow 8k+17 ],\\

c\{1,2,8^{8k+20}\}=\hspace{-0.3cm} &[0\rightarrow 8k+16, 1\rightarrow 8k+17,
2\rightarrow 8k+18, 3\rightarrow 8k+19, 4\rightarrow 8k+20, 5\rightarrow 8k+5,\\
& 8k+6\rightarrow 6, 8k+21, 8k+13, 8k+15\rightarrow 7, 8k+22, 8k+14
],\\[4pt]
\end{array}
\right.$$
\end{footnotesize}

\begin{footnotesize}
$$\left.\begin{array}{rl}
c\{1,2^2,8^{8k+13}\}=\hspace{-0.3cm} &[0\rightarrow 8k+16 ,7\rightarrow 8k+15 ,
6\rightarrow 8k+14 ,5\rightarrow 8k+13 ,4\rightarrow 8k+12, 8k+10\rightarrow \\
& 2, 3\rightarrow 8k+11, 8k+9\rightarrow 1 ],\\

c\{1,2^2,8^{8k+14}\}=\hspace{-0.3cm} &[0\rightarrow 8k+16 ,6\rightarrow 8k+14 ,
4\rightarrow 8k+12, 8k+10\rightarrow 2, 1 \rightarrow 8k+17 , 7\rightarrow 8k+15, \\
& 5 \rightarrow 8k+13, 8k+11\rightarrow 3],\\

c\{1,2^2,8^{8k+15}\}=\hspace{-0.3cm} &[0\rightarrow 8k+16 ,5\rightarrow 8k+13 ,
8k+15\rightarrow 7,  8k+18 \rightarrow 2 , 4\rightarrow 8k+12, 8k+11\rightarrow 3,\\
& 8k+14 \rightarrow 6 , 8k+17\rightarrow  1 ],\\

c\{1,2^2,8^{8k+16}\}=\hspace{-0.3cm} &[0\rightarrow 8k+16 , 4\rightarrow 8k+12,
8k+14\rightarrow 6 , 8k+18\rightarrow 2, 3\rightarrow 8k+19 , 7\rightarrow \\
& 8k+15, 8k+13\rightarrow 5 , 8k+17\rightarrow 1 ],\\

c\{1,2^2,8^{8k+17}\}=\hspace{-0.3cm} &[0\rightarrow 8k+16 , 3\rightarrow 8k+19 ,
6\rightarrow 8k+14, 8k+13\rightarrow 5, 7\rightarrow 8k+15 , 2\rightarrow \\
& 8k+18, 8k+20\rightarrow 4 , 8k+17\rightarrow 1 ],\\

c\{1,2^2,8^{8k+18}\}=\hspace{-0.3cm} &[0\rightarrow 8k+16 , 2\rightarrow 8k+18 ,
4\rightarrow 8k+20 , 6\rightarrow 8k+14, 8k+13\rightarrow 5, 7\rightarrow 8k+15 ,\\
& 1\rightarrow 8k+17 , 3\rightarrow 8k+19, 8k+21 ],\\

c\{1,2^2,8^{8k+19}\}=\hspace{-0.3cm} &[0\rightarrow 8k+16 , 1\rightarrow 8k+17 ,
2\rightarrow 8k+18 , 3\rightarrow 8k+19 , 4\rightarrow 8k+20, 8k+21 , \\
& 6\rightarrow
8k+22, 7, 5\rightarrow 8k+13, 8k+15\rightarrow 15 ],\\[4pt]

c\{1,2^3,8^{8k+12}\}=\hspace{-0.3cm} &[0\rightarrow 8k+16 ,  7\rightarrow 8k+15 ,
6\rightarrow 8k+14 ,  5\rightarrow 8k+13 ,
8k+11 \rightarrow 3, 1\rightarrow 8k+9, \\
& 8k+10 \rightarrow 2, 4\rightarrow 8k+12],\\

c\{1,2^3,8^{8k+13}\}=\hspace{-0.3cm} &[0\rightarrow 8k+16 ,  6\rightarrow 8k+14 ,
4\rightarrow 8k+12 ,  2\rightarrow 8k+10 ,  8k+11\rightarrow 3,
1\rightarrow 8k+17 , \\
& 8k+15\rightarrow 7, 5\rightarrow 8k+13 ],\\

c\{1,2^3,8^{8k+14}\}=\hspace{-0.3cm} &[0\rightarrow 8k+16 ,  5\rightarrow 8k+13 ,
2\rightarrow 8k+18 ,  7\rightarrow 8k+15, 8k+17\rightarrow 1,
3\rightarrow 8k+11, \\
& 8k+12\rightarrow 4, 6\rightarrow 8k+14 ],\\

c\{1,2^3,8^{8k+15}\}=\hspace{-0.3cm} &[0\rightarrow 8k+16 ,  4\rightarrow 8k+12,
8k+14\rightarrow 6 ,  8k+18\rightarrow 2, 1\rightarrow 8k+17,
8k+19\rightarrow 3 ,\\
& 8k+15\rightarrow 7, 5\rightarrow 8k+13 ],\\

c\{1,2^3,8^{8k+16}\}=\hspace{-0.3cm} &[0\rightarrow 8k+16 ,  3\rightarrow 8k+19 ,
6\rightarrow 8k+14 ,  1\rightarrow 8k+17, 8k+18\rightarrow 2,
4\rightarrow 8k+20 , \\
& 7, 5\rightarrow 8k+13, 8k+15\rightarrow 15 ],\\

c\{1,2^3,8^{8k+17}\}=\hspace{-0.3cm} &[0\rightarrow 8k+16 ,  2\rightarrow 8k+18 ,
4\rightarrow 8k+20 ,  6\rightarrow 8k+14, 8k+13\rightarrow 5 ,
8k+19,\\
& 8k+21 ,  7\rightarrow 8k+15, 8k+17\rightarrow 1, 3\rightarrow 8k+11 ],\\

c\{1,2^3,8^{8k+18}\}=\hspace{-0.3cm} &[0\rightarrow 8k+16 ,  1\rightarrow 8k+17 ,
2\rightarrow 8k+18 ,  3\rightarrow 8k+19, 4\rightarrow 8k+20,
8k+22 ,  7, \\
& 5, 13, 15\rightarrow 8k+15, 8k+14\rightarrow 6 ,  8k+21\rightarrow
21 ],\\[4pt]

c\{1,2^4,8^{8k+11}\}=\hspace{-0.3cm} &[0\rightarrow 8k+16 ,  7\rightarrow 8k+15 ,
6\rightarrow 8k+14 ,  8k+12\rightarrow 4 , 2\rightarrow 8k+10, 8k+9\rightarrow 1,\\
&  3\rightarrow 8k+11, 8k+13\rightarrow 5 ],\\

c\{1,2^4,8^{8k+12}\}=\hspace{-0.3cm} &[0\rightarrow 8k+16 ,  6\rightarrow 8k+14 ,
4\rightarrow 8k+12 ,  8k+10\rightarrow 2, 1\rightarrow 8k+17 ,  8k+15\rightarrow
7, \\
& 5\rightarrow 8k+13, 8k+11\rightarrow 3 ],\\

c\{1,2^4,8^{8k+13}\}=\hspace{-0.3cm} &[0\rightarrow 8k+16 ,  5\rightarrow 8k+13 ,
2\rightarrow 8k+18 ,  7\rightarrow 8k+15, 8k+17\rightarrow 1, 3\rightarrow 8k+11, \\
& 8k+12,8k+14\rightarrow 6, 4\rightarrow 8k+4 ],\\

c\{1,2^4,8^{8k+14}\}=\hspace{-0.3cm} &[0\rightarrow 8k+16 ,  4\rightarrow 8k+12,
8k+14\rightarrow 6 ,  8k+18\rightarrow 2, 1\rightarrow 8k+17, 8k+19 ,\\
& 7\rightarrow 8k+15, 8k+13\rightarrow 5, 3\rightarrow 8k+11 ],\\

c\{1,2^4,8^{8k+15}\}=\hspace{-0.3cm} &[0\rightarrow 8k+16 ,  3\rightarrow 8k+19 ,
6\rightarrow 8k+14 ,  1\rightarrow 8k+17, 8k+15\rightarrow 7, 5\rightarrow 8k+13,\\
& 8k+12\rightarrow 4, 2\rightarrow 8k+18, 8k+20 ],\\

c\{1,2^4,8^{8k+16}\}=\hspace{-0.3cm} &[0\rightarrow 8k+16 ,  2\rightarrow 8k+18 ,
4\rightarrow 8k+20 ,  6\rightarrow 8k+14, 8k+13\rightarrow 5, 7\rightarrow 8k+15 , \\
& 1\rightarrow 8k+9, 8k+11\rightarrow 3 ,  8k+17, 8k+19, 8k+21 ],\\

c\{1,2^4,8^{8k+17}\}=\hspace{-0.3cm} &[0\rightarrow 8k+16 ,  1\rightarrow 8k+17 ,
2\rightarrow 8k+18, 8k+20\rightarrow 4, 3\rightarrow 8k+19, 8k+21,\\
& 6\rightarrow 8k+22 , 7, 5\rightarrow 8k+13, 8k+15\rightarrow 15 ],\\[4pt]

c\{1,2^5,8^{8k+10}\}=\hspace{-0.3cm} &[0\rightarrow 8k+16 ,  7\rightarrow 8k+15 ,
6\rightarrow 8k+14, 8k+12\rightarrow 4, 2\rightarrow 8k+10, 8k+9\rightarrow 1,\\
& 3, 5\rightarrow 8k+13, 8k+11\rightarrow 11 ],\\

c\{1,2^5,8^{8k+11}\}=\hspace{-0.3cm} &[0\rightarrow 8k+16 ,  6\rightarrow 8k+14 ,
4\rightarrow 8k+12, 2\rightarrow 8k+10, 8k+9\rightarrow 1, 3\rightarrow 8k+11,\\
& 8k+13,8k+15, 8k+17, 7\rightarrow 8k+7, 8k+5\rightarrow 5 ],\\

c\{1,2^5,8^{8k+12}\}=\hspace{-0.3cm} &[0\rightarrow 8k+16 ,  5\rightarrow 8k+13 ,
2\rightarrow 8k+10, 8k+12\rightarrow 4, 6\rightarrow 8k+14 ,  3\rightarrow 8k+11,\\
& 8k+9\rightarrow 1 , 8k+18, 8k+17, 8k+15\rightarrow 7 ],\\

c\{1,2^5,8^{8k+13}\}=\hspace{-0.3cm} &[0\rightarrow 8k+16 ,  4\rightarrow 8k+12,
8k+14\rightarrow 6 ,  8k+18\rightarrow 2, 3\rightarrow 8k+11,8k+9\rightarrow 1 , \\
& 8k+13\rightarrow 5,7\rightarrow 8k+15, 8k+17, 8k+19 ],\\

c\{1,2^5,8^{8k+14}\}=\hspace{-0.3cm} &[0\rightarrow 8k+16 ,  3\rightarrow 8k+19 ,
6\rightarrow 8k+14 ,  1\rightarrow 8k+17, 4, 2\rightarrow 8k+10, \\
& 8k+12\rightarrow 12,13\rightarrow 8k+13, 8k+15\rightarrow 7, 5 ,  8k+18, 8k+20 ],\\

c\{1,2^5,8^{8k+15}\}=\hspace{-0.3cm} &[0\rightarrow 8k+16 ,  2\rightarrow 8k+18 ,
4\rightarrow 8k+20 ,  6\rightarrow 8k+14, 8k+15\rightarrow 7, 9\rightarrow 8k+17,\\
& 8k+19, 5\rightarrow 8k+13, 8k+11\rightarrow 3, 1 ,  8k+21 ],\\

c\{1,2^5,8^{8k+16}\}=\hspace{-0.3cm} &[0\rightarrow 8k+16 ,  1\rightarrow 8k+17 ,
2\rightarrow 8k+18, 8k+20\rightarrow 4, 6\rightarrow 8k+6, 8k+5\rightarrow
5,\\
& 3\rightarrow 8k+19, 8k+21, 8k+13, 8k+15\rightarrow 7, 8k+22, 8k+14
],\\[4pt]
\end{array}
\right.$$
\end{footnotesize}

\begin{footnotesize}
$$\left.\begin{array}{rl}

c\{1^2,2,8^{8k+13}\}=\hspace{-0.3cm} &[0\rightarrow 8k+16 ,  7\rightarrow 8k+15 ,
6\rightarrow 8k+14 ,  5\rightarrow 8k+13 ,  4\rightarrow 8k+12 ,  3\rightarrow 8k+11,\\
& 8k+9,8k+10\rightarrow 2, 1\rightarrow 8k+1 ],\\

c\{1^2,2,8^{8k+14}\}=\hspace{-0.3cm} &[0\rightarrow 8k+16 ,  6\rightarrow 8k+14 ,
4\rightarrow 8k+12, 8k+13\rightarrow 5 ,  8k+15\rightarrow 7 ,  8k+17\rightarrow 1,\\
& 3\rightarrow 8k+11, 8k+10\rightarrow 2 ],\\

c\{1^2,2,8^{8k+15}\}=\hspace{-0.3cm} &[0\rightarrow 8k+16 ,  5\rightarrow 8k+13,
8k+15\rightarrow 7 ,  8k+18\rightarrow 2, 1\rightarrow 8k+17 ,  6\rightarrow 8k+14 , \\
& 3\rightarrow 8k+11, 8k+12\rightarrow 4 ],\\

c\{1^2,2,8^{8k+16}\}=\hspace{-0.3cm} &[0\rightarrow 8k+16 ,  4\rightarrow 8k+12,
8k+13\rightarrow 5 ,  8k+17\rightarrow 1, 3\rightarrow 8k+19 ,  7\rightarrow 8k+15,\\
&8k+14\rightarrow 6,8k+18\rightarrow 2 ],\\

c\{1^2,2,8^{8k+17}\}=\hspace{-0.3cm} &[0\rightarrow 8k+16 ,  3\rightarrow 8k+19 ,
6\rightarrow 8k+14 ,  1\rightarrow 8k+17, 8k+18\rightarrow 2 ,  8k+15, \\
& 8k+13\rightarrow 5,4\rightarrow 8k+20, 7\rightarrow 8k+7 ],\\

c\{1^2,2,8^{8k+18}\}=\hspace{-0.3cm} &[0\rightarrow 8k+16 ,  2\rightarrow 8k+18 ,
8k+19\rightarrow 3 ,  8k+17\rightarrow 1 ,  8k+15\rightarrow 7 ,  8k+21\rightarrow 5,\\
& 4\rightarrow 8k+12, 8k+14\rightarrow 6, 8k+20 ],\\

c\{1^2,2,8^{8k+19}\}=\hspace{-0.3cm} &[0\rightarrow 8k+16 ,  1\rightarrow 8k+17 ,
2\rightarrow 8k+18 ,  3\rightarrow 8k+19 ,  4\rightarrow 8k+20, 8k+22 ,  \\
& 7\rightarrow 8k+15, 8k+14\rightarrow 6, 5\rightarrow 8k+21 ],\\[4pt]

c\{1^2,2^2,8^{8k+12}\}=\hspace{-0.3cm} &[0\rightarrow 8k+16 ,  7\rightarrow 8k+15 ,
6\rightarrow 8k+14 ,  5\rightarrow 8k+13, 8k+12\rightarrow 4, 2\rightarrow 8k+10,\\
& 8k+11\rightarrow 3, 1\rightarrow 8k+9 ],\\

c\{1^2,2^2,8^{8k+13}\}=\hspace{-0.3cm} &[0\rightarrow 8k+16 ,  6\rightarrow 8k+14 ,
4\rightarrow 8k+12, 8k+13\rightarrow 5, 8k+15\rightarrow 7 ,  8k+17, \\
& 1\rightarrow 8k+9, 8k+11\rightarrow 3, 2\rightarrow 8k+10 ],\\

c\{1^2,2^2,8^{8k+14}\}=\hspace{-0.3cm} &[0\rightarrow 8k+16 ,  5\rightarrow 8k+13,
8k+15\rightarrow 7 ,  8k+18\rightarrow 2, 1\rightarrow 8k+17 ,  6\rightarrow 8k+14,\\
& 8k+12\rightarrow 4, 3\rightarrow 8k+11 ],\\

c\{1^2,2^2,8^{8k+15}\}=\hspace{-0.3cm} &[0\rightarrow 8k+16 ,  4\rightarrow 8k+12,
8k+14\rightarrow 6, 5\rightarrow 8k+13 ,  1\rightarrow 8k+17, 8k+15\rightarrow 7 , \\
& 8k+19\rightarrow 3, 2\rightarrow 8k+18 ],\\

c\{1^2,2^2,8^{8k+16}\}=\hspace{-0.3cm} &[0\rightarrow 8k+16 ,  3\rightarrow 8k+19 ,
6\rightarrow 8k+14 ,  1\rightarrow 8k+17, 8k+18\rightarrow 10, 12\rightarrow 8k+20 , \\
& 7\rightarrow 8k+15, 2, 4, 5\rightarrow 8k+13 ],\\

c\{1^2,2^2,8^{8k+17}\}=\hspace{-0.3cm} &[0\rightarrow 8k+16 ,  2\rightarrow 8k+18 ,
4\rightarrow 8k+20, 8k+19\rightarrow 3 ,  8k+17\rightarrow 1, 8k+21 , \\
& 7\rightarrow 8k+15, 8k+13\rightarrow 5, 6\rightarrow 8k+14 ],\\

c\{1^2,2^2,8^{8k+18}\}=\hspace{-0.3cm} &[0\rightarrow 8k+16 ,  1\rightarrow 8k+17 ,
2\rightarrow 8k+18, 8k+20\rightarrow 4 ,  8k+19\rightarrow 3, 5, 6 , \\
& 8k+21\rightarrow  13,  14\rightarrow 8k+22 ,  7\rightarrow 8k+15 ],\\[4pt]

c\{1^2,2^3,8^{8k+11}\}=\hspace{-0.3cm} & [0\rightarrow 8k+16 ,  7\rightarrow 8k+15 ,
6\rightarrow 8k+14 ,  5\rightarrow 8k+13, 8k+12\rightarrow 4, 2\rightarrow 8k+10,\\
& 8k+11, 8k+9\rightarrow 1, 3\rightarrow 8k+3 ],\\

c\{1^2,2^3,8^{8k+12}\}=\hspace{-0.3cm} &[0\rightarrow 8k+16 ,  6\rightarrow 8k+14 ,
4\rightarrow 8k+12, 8k+13\rightarrow 5 ,  8k+15\rightarrow 7, 9\rightarrow 8k+17, \\
& 1, 3\rightarrow 8k+11, 8k+10\rightarrow 2 ],\\

c\{1^2,2^3,8^{8k+13}\}=\hspace{-0.3cm} &[0\rightarrow 8k+16 ,  5\rightarrow 8k+13,
8k+15\rightarrow 7 ,  8k+18\rightarrow 2, 1\rightarrow 8k+9, 8k+11\rightarrow 3,\\
& 4\rightarrow 8k+12, 8k+14\rightarrow 6 ,  8k+17 ],\\

c\{1^2,2^3,8^{8k+14}\}=\hspace{-0.3cm} &[0\rightarrow 8k+16 ,  4\rightarrow 8k+12,
8k+14\rightarrow 6, 5\rightarrow 8k+13, 8k+15\rightarrow 7 ,  8k+19\rightarrow 3,\\
& 1\rightarrow 8k+17, 8k+18\rightarrow 2 ],\\

c\{1^2,2^3,8^{8k+15}\}=\hspace{-0.3cm} &[0\rightarrow 8k+16 ,  3\rightarrow 8k+19,
8k+17\rightarrow 1, 2\rightarrow 8k+18 ,  5\rightarrow 8k+13, 8k+15, \\
& 8k+14\rightarrow 6, 4\rightarrow 8k+20, 7\rightarrow 8k+7 ],\\

c\{1^2,2^3,8^{8k+16}\}=\hspace{-0.3cm} &[0\rightarrow 8k+16 ,  2\rightarrow 8k+18 ,
4\rightarrow 8k+20, 8k+19\rightarrow 3, 5\rightarrow 8k+21, \\
& 1\rightarrow 8k+17, 8k+15\rightarrow 7,
6\rightarrow 8k+14 ],\\

c\{1^2,2^3,8^{8k+17}\}=\hspace{-0.3cm} &[0\rightarrow 8k+16 ,  1\rightarrow 8k+17,
8k+15\rightarrow 7, 5\rightarrow 8k+21, 8k+22\rightarrow 6, 4\rightarrow 8k+20,\\
& 8k+19\rightarrow 3 ,  8k+18\rightarrow 2 ],\\[4pt]

c\{1^2,2^4,8^{8k+10}\}=\hspace{-0.3cm} &[0\rightarrow 8k+16 ,  7\rightarrow 8k+15 ,
6\rightarrow 8k+14 ,  5\rightarrow 8k+13, 8k+11\rightarrow 3, 1, 2, \\
& 4\rightarrow 8k+12, 8k+10\rightarrow 10, 9\rightarrow 8k+9 ],\\

c\{1^2,2^4,8^{8k+11}\}=\hspace{-0.3cm} &[0\rightarrow 8k+16 ,  6\rightarrow 8k+14 ,
4\rightarrow 8k+12, 8k+13\rightarrow 5, 7\rightarrow 8k+15, 8k+17,\\
& 1\rightarrow 8k+9, 8k+11\rightarrow 3, 2\rightarrow 8k+10 ],\\

c\{1^2,2^4,8^{8k+12}\}=\hspace{-0.3cm} &[0\rightarrow 8k+16 ,  5\rightarrow 8k+13,
8k+15\rightarrow 7 ,  8k+18\rightarrow 2, 1, 3, 8k+14\rightarrow 6, \\
& 4\rightarrow 8k+12, 8k+11\rightarrow 11, 9\rightarrow 8k+17 ],\\

c\{1^2,2^4,8^{8k+13}\}=\hspace{-0.3cm} &[0\rightarrow 8k+16 ,  4\rightarrow 8k+12,
8k+13\rightarrow 5, 7\rightarrow 8k+15 ,  3\rightarrow 8k+11, \\
& 8k+9\rightarrow 1, 8k+19,8k+17, 8k+18 ,  6\rightarrow 8k+14 ,  2\rightarrow 8k+10 ],\\

c\{1^2,2^4,8^{8k+14}\}=\hspace{-0.3cm} &[0\rightarrow 8k+16 ,  3\rightarrow 8k+19,
8k+17\rightarrow 1, 2\rightarrow 8k+18, 8k+20\rightarrow 4, \\
& 6\rightarrow 8k+14, 8k+15\rightarrow 7,
5\rightarrow 8k+13 ],\\
\end{array}
\right.$$
\end{footnotesize}

\begin{footnotesize}
$$\left.\begin{array}{rl}
c\{1^2,2^4,8^{8k+15}\}=\hspace{-0.3cm} &[0\rightarrow 8k+16 ,  2\rightarrow 8k+18,
8k+20\rightarrow 4, 5\rightarrow 8k+21, 8k+19\rightarrow 11, \\
& 9\rightarrow 8k+17 ,  3, 1, 8k+15\rightarrow 7, 6\rightarrow 8k+14 ],\\

c\{1^2,2^4,8^{8k+16}\}=\hspace{-0.3cm} &[0\rightarrow 8k+16 ,  1\rightarrow 8k+17,
8k+15\rightarrow 7, 5\rightarrow 8k+21, 8k+22\rightarrow 6, \\
& 4\rightarrow 8k+20, 8k+18\rightarrow 2, 3\rightarrow 8k+19 ],\\[4pt]

c\{1^3,2,8^{8k+12}\}=\hspace{-0.3cm} & [0\rightarrow 8k+16 ,  7\rightarrow 8k+15 ,
6\rightarrow 8k+14 ,  5\rightarrow 8k+13, 8k+12\rightarrow 4, \\
& 3\rightarrow 8k+11, 8k+9\rightarrow 1, 2\rightarrow 8k+10 ],\\

c\{1^3,2,8^{8k+13}\}=\hspace{-0.3cm} & [0\rightarrow 8k+16 ,  6\rightarrow 8k+14 ,
4\rightarrow 8k+12, 8k+13\rightarrow 5 ,  8k+15\rightarrow 7 ,  8k+17,\\
& 1\rightarrow 8k+9, 8k+10\rightarrow 2, 3\rightarrow 8k+11 ],\\

c\{1^3,2,8^{8k+14}\}=\hspace{-0.3cm} & [0\rightarrow 8k+16 ,  5\rightarrow 8k+13,
8k+15\rightarrow 7 ,  8k+18\rightarrow 2, 1\rightarrow 8k+17 , \\
& 6\rightarrow 8k+14, 3, 4\rightarrow 8k+12, 8k+11\rightarrow 11 ],\\

c\{1^3,2,8^{8k+15}\}=\hspace{-0.3cm} & [0\rightarrow 8k+16 ,  4\rightarrow 8k+12,
8k+14\rightarrow 6, 5\rightarrow 8k+13 ,  1\rightarrow 8k+17,\\
& 8k+18\rightarrow 2, 3\rightarrow 8k+19 ,
7\rightarrow 8k+15 ],\\

c\{1^3,2,8^{8k+16}\}=\hspace{-0.3cm} & [0\rightarrow 8k+16 ,  3\rightarrow 8k+19,
8k+18\rightarrow 2, 1\rightarrow 8k+17 ,  4\rightarrow 8k+20 ,  \\
& 7\rightarrow 8k+15, 8k+13\rightarrow 5,
6\rightarrow 8k+14 ],\\

c\{1^3,2,8^{8k+17}\}=\hspace{-0.3cm} & [0\rightarrow 8k+16 ,  2\rightarrow 8k+18 ,
4\rightarrow 8k+20, 8k+19\rightarrow 3 ,  8k+17\rightarrow 1 , \\
&  8k+21\rightarrow 5, 6\rightarrow 8k+14,
8k+15\rightarrow 7 ],\\

c\{1^3,2,8^{8k+18}\}=\hspace{-0.3cm} & [0\rightarrow 8k+16 ,  1\rightarrow 8k+17 ,
2\rightarrow 8k+18 ,  3\rightarrow 8k+19 ,  4\rightarrow 8k+20, 8k+21, \\
& 8k+22\rightarrow 6,
7\rightarrow 8k+15, 8k+13\rightarrow 5 ],\\[4pt]

c\{1^3,2^2,8^{8k+11}\}=\hspace{-0.3cm} & [0\rightarrow 8k+16 ,  7\rightarrow 8k+15 ,
6\rightarrow 8k+14, 8k+13\rightarrow 5, 4\rightarrow 8k+12,\\
 & 8k+10\rightarrow 2, 3\rightarrow 8k+11,
8k+9\rightarrow 1 ],\\

c\{1^3,2^2,8^{8k+12}\}=\hspace{-0.3cm} & [0\rightarrow 8k+16 ,  6\rightarrow 8k+14 ,
4\rightarrow 8k+12, 8k+13\rightarrow 5, 3\rightarrow 8k+11, \\
& 8k+10\rightarrow 2,
1\rightarrow 8k+17,
8k+15\rightarrow 7 ],\\

c\{1^3,2^2,8^{8k+13}\}=\hspace{-0.3cm} & [0\rightarrow 8k+16 ,  5\rightarrow 8k+13,
8k+15\rightarrow 7 ,  8k+18\rightarrow 2, 1\rightarrow 8k+17 ,  6\rightarrow 8k+14,\\
&
8k+12,
 8k+11\rightarrow 3, 4\rightarrow 8k+4 ],\\

c\{1^3,2^2,8^{8k+14}\}=\hspace{-0.3cm} & [0\rightarrow 8k+16 ,  4\rightarrow 8k+12,
8k+14\rightarrow 6, 5\rightarrow 8k+13 ,  1\rightarrow 8k+17, 8k+15\rightarrow 7 ,\\
& 8k+19, 8k+18\rightarrow 2, 3\rightarrow 8k+11 ] ,\\

c\{1^3,2^2,8^{8k+15}\}=\hspace{-0.3cm} & [0\rightarrow 8k+16 ,  3\rightarrow 8k+19,
8k+17\rightarrow 1, 2\rightarrow 8k+18 ,  5\rightarrow 8k+13, 8k+12\rightarrow 4,\\
& 6\rightarrow 8k+14,
 8k+15\rightarrow 7 ,  8k+20 ],\\

c\{1^3,2^2,8^{8k+16}\}=\hspace{-0.3cm} & [0\rightarrow 8k+16 ,  2\rightarrow 8k+18 ,
4\rightarrow 8k+20, 8k+19\rightarrow 3 ,  8k+17\rightarrow 1, 8k+21\rightarrow 5, 7,\\
&  6\rightarrow 8k+14, 8k+15\rightarrow 15 ],\\

c\{1^3,2^2,8^{8k+17}\}=\hspace{-0.3cm} & [0\rightarrow 8k+16 ,  1\rightarrow 8k+17 ,
2\rightarrow 8k+18, 8k+20\rightarrow 4 ,  8k+19\rightarrow 3, 5\rightarrow 8k+13,\\
& 8k+14, 8k+15\rightarrow 7, 8k+22, 8k+21 ,  6\rightarrow 8k+6 ],\\[4pt]

c\{1^3,2^3,8^{8k+10}\}=\hspace{-0.3cm} & [0\rightarrow 8k+16 ,  7\rightarrow 8k+15 ,
6\rightarrow 8k+14, 8k+13\rightarrow 5, 4\rightarrow 8k+12, 8k+10\rightarrow 2,\\
& 3, 1\rightarrow 8k+9, 8k+11\rightarrow 11 ],\\

c\{1^3,2^3,8^{8k+11}\}=\hspace{-0.3cm} & [0\rightarrow 8k+16 ,  6\rightarrow 8k+14 ,
4\rightarrow 8k+12, 8k+13\rightarrow 5, 3\rightarrow 8k+11 ,  1, 2\rightarrow 8k+10,\\
&
8k+9\rightarrow 9, 7\rightarrow 8k+15, 8k+17 ],\\

c\{1^3,2^3,8^{8k+12}\}=\hspace{-0.3cm} & [0\rightarrow 8k+16 ,  5\rightarrow 8k+13,
8k+15\rightarrow 7 ,  8k+18\rightarrow 2, 1\rightarrow 8k+17 ,  6\rightarrow 8k+6,\\
& 8k+4\rightarrow 4, 3\rightarrow 8k+11, 8k+12, 8k+14 ],\\

c\{1^3,2^3,8^{8k+13}\}=\hspace{-0.3cm} & [0\rightarrow 8k+16 ,  4\rightarrow 8k+12,
8k+14\rightarrow 6, 5\rightarrow 8k+13, 8k+15\rightarrow 7, 8k+19, \\
& 1\rightarrow 8k+17,
8k+18\rightarrow 2, 3\rightarrow 8k+11 ],\\

c\{1^3,2^3,8^{8k+14}\}=\hspace{-0.3cm} & [0\rightarrow 8k+16 ,  3\rightarrow 8k+19,
8k+17\rightarrow 1, 2\rightarrow 8k+18, 8k+20\rightarrow 4, \\
& 5\rightarrow 8k+13,
8k+15\rightarrow 7, 6\rightarrow 8k+14 ],\\

c\{1^3,2^3,8^{8k+15}\}=\hspace{-0.3cm} & [0\rightarrow 8k+16 ,  2\rightarrow 8k+18,
4\rightarrow 8k+20, 8k+19\rightarrow 3, 5\rightarrow 8k+13, 8k+14\rightarrow 6,\\
& 7\rightarrow 8k+15, 8k+17\rightarrow 1, 8k+21 ],\\

c\{1^3,2^3,8^{8k+16}\}=\hspace{-0.3cm} & [0\rightarrow 8k+16 ,  1\rightarrow 8k+17 ,
2\rightarrow 8k+18, 8k+20\rightarrow 4, 3\rightarrow 8k+19, 8k+21,\\
& 8k+22, 7\rightarrow
8k+15, 8k+13\rightarrow 5, 6\rightarrow 8k+14 ],\\[4pt]

c\{1^4,2,8^{8k+11}\}=\hspace{-0.3cm} & [0\rightarrow 8k+16 ,  7\rightarrow 8k+15 ,
6\rightarrow 8k+14 ,  5\rightarrow 8k+13 ,  8k+12\rightarrow 4, \\
& 3\rightarrow 8k+11,  8k+9, 8k+10 \rightarrow 2, 1\rightarrow 8k+1 ],\\

c\{1^4,2,8^{8k+12}\}=\hspace{-0.3cm} & [0\rightarrow 8k+16 ,  6\rightarrow 8k+14 ,
4\rightarrow 8k+12, 8k+13\rightarrow 5 ,  8k+15\rightarrow 7 ,  \\
& 8k+17\rightarrow 9,10\rightarrow 8k+10, 8k+11\rightarrow 3, 1, 2 ],\\

c\{1^4,2,8^{8k+13}\}=\hspace{-0.3cm} & [0\rightarrow 8k+16 ,  5\rightarrow 8k+13,
8k+14\rightarrow 6, 7\rightarrow 8k+15 ,  4\rightarrow 8k+12 ,\\
& 1\rightarrow 8k+9,
8k+11\rightarrow 3, 2\rightarrow 8k+18, 8k+17 ],\\
\end{array}
\right.$$
\end{footnotesize}

\begin{footnotesize}
$$\left.\begin{array}{rl}
c\{1^4,2,8^{8k+14}\}=\hspace{-0.3cm} & [0\rightarrow 8k+16 ,  4\rightarrow 8k+12,
8k+14\rightarrow 6, 5\rightarrow 8k+13 ,  1\rightarrow 8k+17, 8k+18,\\
& 8k+19 ,
7\rightarrow 8k+15 ,  3\rightarrow 8k+11, 8k+10\rightarrow 2 ],\\

c\{1^4,2,8^{8k+15}\}=\hspace{-0.3cm} & [0\rightarrow 8k+16 ,  3\rightarrow 8k+19,
8k+18\rightarrow 2, 1\rightarrow 8k+17, 8k+15\rightarrow 7 , \\
& 8k+20\rightarrow 4,
5\rightarrow 8k+13, 8k+14\rightarrow 6 ] ,\\

c\{1^4,2,8^{8k+16}\}=\hspace{-0.3cm} & [0\rightarrow 8k+16 ,  2\rightarrow 8k+18 ,
4\rightarrow 8k+20, 8k+19\rightarrow 3 ,  8k+17\rightarrow 1 , \\
& 8k+21\rightarrow 5, 6,7\rightarrow 8k+15 ,  8k+14\rightarrow 14 ],\\

c\{1^4,2,8^{8k+17}\}=\hspace{-0.3cm} & [0\rightarrow 8k+16 ,  1\rightarrow 8k+17 ,
2\rightarrow 8k+18 ,  3\rightarrow 8k+19 ,  4\rightarrow 8k+12, \\
& 8k+13\rightarrow 5,
6\rightarrow 8k+14, 8k+15\rightarrow 7 ,  8k+22, 8k+20, 8k+21 ],\\[4pt]

c\{1^4,2^2,8^{8k+10}\}=\hspace{-0.3cm} & [0\rightarrow 8k+16 ,  7\rightarrow 8k+15 ,
6\rightarrow 8k+14 ,  5\rightarrow 8k+5, 8k+4\rightarrow 4, \\
& 2\rightarrow 8k+10,
8k+9\rightarrow 1, 3\rightarrow 8k+11, 8k+12, 8k+13 ],\\

c\{1^4,2^2,8^{8k+11}\}=\hspace{-0.3cm} & [0\rightarrow 8k+16 ,  6\rightarrow 8k+14 ,
4\rightarrow 8k+12, 8k+13\rightarrow 5 ,  8k+15\rightarrow 7 ,  8k+17,\\
& 1\rightarrow
8k+1, 8k+2\rightarrow 2, 3\rightarrow 8k+11, 8k+9, 8k+10 ],\\

c\{1^4,2^2,8^{8k+12}\}=\hspace{-0.3cm} & [0\rightarrow 8k+16 ,  5\rightarrow 8k+13,
8k+14\rightarrow 6, 7\rightarrow 8k+15 ,  4\rightarrow 8k+12, 1, 3, \\
& 2\rightarrow 8k+18,
8k+17\rightarrow 9, 11\rightarrow 8k+11 ],\\

c\{1^4,2^2,8^{8k+13}\}=\hspace{-0.3cm} & [0\rightarrow 8k+16 ,  4\rightarrow 8k+12,
8k+14\rightarrow 6, 5\rightarrow 8k+13 ,  1\rightarrow 8k+17,\\
& 8k+15\rightarrow 7 ,
8k+19, 8k+18, 8k+10, 8k+11\rightarrow 3, 2\rightarrow 8k+2 ],\\

c\{1^4,2^2,8^{8k+14}\}=\hspace{-0.3cm} & [0\rightarrow 8k+16 ,  3\rightarrow 8k+19,
8k+17\rightarrow 1, 2\rightarrow 8k+18, 8k+20\rightarrow 4, \\
& 5\rightarrow 8k+13,
8k+14\rightarrow 6, 7\rightarrow 8k+15 ],\\

c\{1^4,2^2,8^{8k+15}\}=\hspace{-0.3cm} & [0\rightarrow 8k+16 ,  2\rightarrow 8k+18 ,
4\rightarrow 8k+20, 8k+19\rightarrow 3 ,  8k+17\rightarrow 1,\\
& 8k+21\rightarrow 13,
14\rightarrow 8k+14, 8k+15\rightarrow 7, 5, 6 ],\\

c\{1^4,2^2,8^{8k+16}\}=\hspace{-0.3cm} & [0\rightarrow 8k+16 ,  1\rightarrow 8k+17 ,
2\rightarrow 8k+18, 8k+20\rightarrow 4 ,  8k+19\rightarrow 3, \\
& 5\rightarrow 8k+5,
8k+6\rightarrow 6, 7\rightarrow 8k+15, 8k+14, 8k+22, 8k+21, 8k+13
],\\[4pt]

c\{1^5,2,8^{8k+10}\}=\hspace{-0.3cm} & [0\rightarrow 8k+16 ,  7\rightarrow 8k+15 ,
6\rightarrow 8k+14, 8k+13\rightarrow 5, 4\rightarrow 8k+12, 8k+11 , \\
& 2\rightarrow 8k+2,
8k+3\rightarrow 3, 1\rightarrow 8k+9, 8k+10 ],\\

c\{1^5,2,8^{8k+11}\}=\hspace{-0.3cm} & [0\rightarrow 8k+16 ,  6\rightarrow 8k+14 ,
4\rightarrow 8k+12, 8k+13\rightarrow 5 ,  8k+15\rightarrow 7 ,  8k+17, \\
& 1\rightarrow
8k+1, 8k+2\rightarrow 2, 3\rightarrow 8k+11, 8k+10, 8k+9 ],\\

c\{1^5,2,8^{8k+12}\}=\hspace{-0.3cm} & [0\rightarrow 8k+16 ,  5\rightarrow 8k+13,
8k+14\rightarrow 6, 7\rightarrow 8k+15 ,  4\rightarrow 8k+12,\\
& 8k+11\rightarrow 3, 1,
2\rightarrow 8k+18, 8k+17\rightarrow 9 ],\\

c\{1^5,2,8^{8k+13}\}=\hspace{-0.3cm} & [0\rightarrow 8k+16 ,  4\rightarrow 8k+12,
8k+14\rightarrow 6, 5\rightarrow 8k+13 ,  1\rightarrow 8k+17, 8k+18, \\
& 8k+19 ,
7\rightarrow 8k+15, 3, 2\rightarrow 8k+10, 8k+11\rightarrow 11 ],\\

c\{1^5,2,8^{8k+14}\}=\hspace{-0.3cm} & [0\rightarrow 8k+16 ,  3\rightarrow 8k+19,
8k+18\rightarrow 2, 1\rightarrow 8k+17, 8k+15\rightarrow 7,\\
& 6\rightarrow 8k+14,
8k+13\rightarrow 5, 4\rightarrow 8k+20 ],\\

c\{1^5,2,8^{8k+15}\}=\hspace{-0.3cm} & [0\rightarrow 8k+16 ,  2\rightarrow 8k+18 ,
4\rightarrow 8k+20, 8k+19\rightarrow 3 ,  8k+17\rightarrow 1, \\
& 8k+21\rightarrow 13,
14\rightarrow 8k+14, 8k+15\rightarrow 7, 6, 5 ],\\

c\{1^5,2,8^{8k+16}\}=\hspace{-0.3cm} &  [0\rightarrow 8k+16 ,  1\rightarrow 8k+17 ,
2\rightarrow 8k+18 ,  3\rightarrow 8k+19 , 4\rightarrow 8k+20, 8k+22, \\
& 8k+21,
8k+13, 8k+14, 8k+15\rightarrow 7, 6\rightarrow 8k+6, 8k+5\rightarrow 5 ].
\end{array}
\right.$$
\end{footnotesize}
\end{proof}

\subsection*{Acknowledgments}
The authors thank  the anonymous referee whose helpful comments and suggestions
improved the presentation of the paper.

\end{document}